\documentclass[12pt,a4]{amsart}
\usepackage{amsmath}
\usepackage{amssymb}
\usepackage{amsthm}
\usepackage[shortlabels]{enumitem}
\usepackage{url}
\usepackage{color}
\usepackage{dsfont}
\usepackage{epsfig}
\usepackage[backref=page,hidelinks]{hyperref}
\usepackage{setspace}
\usepackage{epstopdf}
\usepackage[numbers,square,sort,comma]{natbib}
\usepackage{comment}
\usepackage{booktabs}
\usepackage{float}
\usepackage{tikz}
\usepackage{algorithm}
\usepackage{stmaryrd}
\usepackage{times}
\numberwithin{equation}{section}

\usepackage[font=small,format=plain,labelfont=bf,up,textfont=it,up]{caption}
\usepackage{subcaption}
\usepackage{geometry}
\theoremstyle{plain}
\newtheorem{theorem}{Theorem}[section]
\newtheorem{proposition}[theorem]{Proposition}

\newtheorem{lemma}[theorem]{Lemma}
\newtheorem{corollary}[theorem]{Corollary}

\theoremstyle{definition}
\newtheorem{definition}[theorem]{Definition}

\newtheorem{example}[theorem]{Example}

\theoremstyle{remark}
\newtheorem{rk}{Remark}[section]
\expandafter\let\expandafter\oldproof\csname\string\proof\endcsname
\let\oldendproof\endproof
\renewenvironment{proof}[1][\proofname]{%
  \oldproof[\noindent\textbf{#1.} ]%
}{\oldendproof}
\newcommand{\1}{\mathds{1}}

\newcommand{\E}{\mathbb{E}}
\newcommand{\mP}{\mathbb{P}}
\newcommand{\be}{\begin{equation}}
\newcommand{\ee}{\end{equation}}
\newcommand{\by}{\begin{eqnarray*}}
\newcommand{\ey}{\end{eqnarray*}}

\renewcommand{\leq}{\leqslant}
\renewcommand{\geq}{\geqslant}
\newcommand{\dx}{\mathtt{d}}
\newcommand{\bx}{\mathtt{b}}

\usepackage{xcolor}
\definecolor{dark-red}{rgb}{0.4,0.15,0.15}
\definecolor{dark-blue}{rgb}{0.15,0.15,0.4}
\definecolor{medium-blue}{rgb}{0,0,0.5}
\hypersetup{
    colorlinks, linkcolor={red},
    citecolor={blue}, urlcolor={blue}
}
\newcommand{\norm}[1]{\left\lVert#1\right\rVert}

\renewcommand{\r}{\mathfrak{r}}
\renewcommand{\Mc}{\mathcal{MC}}
\newcommand{\Sf}{\mathcal{M}}
\newcommand{\M}{\mathcal{M}}
\newcommand{\N}{\mathbb{N}}

\renewcommand{\S}{\mathcal{S}}
\newcommand{\GMc}{\mathcal{GMC}}
\newmuskip\pFqmuskip

\begin{document}
\title{Analysis of non-reversible Markov chains via similarity orbit}
\author{Michael C.H. Choi}
\address{Institute for Data and Decision Analytics, The Chinese University of Hong Kong, Shenzhen, Guangdong, 518172, P.R. China and Shenzhen Institute of Artificial Intelligence and Robotics for Society}
\email{michaelchoi@cuhk.edu.cn}
\author{Pierre Patie}\thanks{The authors would like to thank two anonymous referees for their careful reading and constructive comments that have improved the presentation of the manuscript. This work was partially supported by  NSF Grant DMS-1406599 and ARC IAPAS, a fund of the Communaut\'ee francaise de Belgique. The first author would also like to acknowledge the support of The Chinese University of Hong Kong, Shenzhen grant PF01001143 and the financial support from AIRS - Shenzhen Institute of Artificial Intelligence and Robotics for Society Project 2019-INT002. The second author is grateful  for the hospitality of  the  LMA at the UPPA, where part of this work was completed.}
\address{School of Operations Research and Information Engineering, Cornell University, Ithaca, NY 14853, and, Laboratoire de Math\'ematiques et leurs  applications, UMR CNRS 5142, Universit\'e de Pau et des Pays de l'Adour, Pau, France.}
\email{pp396@cornell.edu}
\date{\today}



\begin{abstract}
	In this paper, we develop an in-depth analysis of non-reversible Markov chains on denumerable state space from a similarity orbit perspective. In particular, we study the class of Markov chains whose transition kernel is in the similarity orbit of a normal transition kernel, such as the one of  birth-death chains or reversible Markov chains. We start by identifying a set of sufficient  conditions for a  Markov chain to belong to the similarity orbit of a birth-death one.  As by-products, we obtain a spectral representation in terms of non-self-adjoint resolutions of identity in the sense of Dunford \cite{Dunford54} and offer a detailed analysis on the convergence rate, separation cutoff and ${\rm{L}}^2$-cutoff of this class of non-reversible Markov chains. We also look into the problem of estimating the integral functionals from discrete observations for this class. In the last part of this paper, we investigate a particular similarity orbit of reversible Markov kernels, that we call the pure birth orbit, and analyze various possibly non-reversible variants of classical birth-death processes in this orbit.
\end{abstract}

\maketitle

\tableofcontents

\section{Introduction}\label{sec:intro}

The spectral theorem of normal operators is undoubtedly a powerful tool to deal with substantial and difficult issues arising in the analysis of Markov chains. The intrusion of spectral theory to the analysis of Markov chains dates back to the long line of work initiated by \cite{LR54} and \cite{KM59} who were among the first to offer a detailed spectral analysis in the direction of reversible birth-death processes. Beyond eigenvalues expansion, the spectral theorem also appears in the study of the rate of convergence to equilibrium, mixing time, eigentime identity, separation cutoff and ${\rm{L}}^2$-cutoff, see e.g.~\cite{AF14,DSC06,CuiMao10,Miclo15,CSC10,LPW09}, to name but a few. It is also central for their statistical estimations  as it is demonstrated by the recent work of  \cite{AC16} for the integral functionals of normal Markov chains.

However, the lack of spectral theorem for non-normal operators gives major difficulties for tackling these fundamental topics in the context of general Markov chains, since the transition kernel $P$ is  a non-normal linear operator in the weighted Hilbert space \[ \ell^2(\pi)=\bigg\{ f: \mathcal{X} \mapsto \mathbb{C}; \:||f||^2_{\pi}=\sum_{x\in \mathcal{X}} |f(x)|^2\pi(x)<\infty\bigg\},\]  with $\pi$ being a reference (invariant or excessive) measure  of $P$ and $\pi(x) > 0$ for all $x \in \mathcal{X}$. Not only the non-reversibility property, or generally the non-normality of $P$, is a generic property from a theoretical perspective, it is also natural and becomes increasingly popular recently in various applications. For instance, non-reversible Markov chains  appear in the study of queueing networks and fluid approximation in \cite{FMM08}, hyperplane rearrangement in \cite{Pike13} and the very recent introduction of non-reversible Metropolis-Hastings and its variants, see e.g.~\cite{RR15,Bie16}.

To overcome the challenge of analyzing non-self-adjoint operators,  a wide variety of intriguing ideas has been elaborated to deal with specific issues. This includes, for example, the dilation concept developed by \cite{Kendall59}, reversiblizations techniques as in \cite{Fill91,Paulin15} or recasting to a weighted-$L^{\infty}$ space \cite{KM03,KM05,KM12}.

In this paper, we propose an alternative remedy by resorting to the algebraic  concept of similarity orbit of normal Markov chains  as defined in Definition \ref{def:sim} below. This identifies a class of transition kernels of Markov chains,  denoted by $\mathcal{S}$, which is a subset of  $\M$, the set of Markov transition kernels  acting on a countable state space $\mathcal{X}$.  We  emphasize that our approach  offers an unifying framework to analyze  all substantial  and classical topics for Markov kernels in $\mathcal S$  that were enumerated above for normal Markov  chains.
 This extends the work by the authors in \cite{Choi-Patie} from  skip-free Markov chains to
  general ones.   It is also in line with the papers by \cite{Miclo2016, Patie-Savov} and \cite{PatieZhao17} for the study of spectral theory of non-reversible Markov processes and  by \cite{FD,ChafaiJoulin13} and \cite{CD16} for birth-death processes, which rely on the  notion of intertwining relationships.
 We proceed by recalling the definition of similarity  orbit as introduced in \cite{Choi-Patie}.

\begin{definition}[Similarity]\label{def:sim}
	We say that the transition kernel $P \in \Sf$ of a Markov chain $X $  is similar to the transition kernel $Q$ of a Markov chain on $\mathcal{X}$, and we write  $P \sim  Q$, if there exists a bounded linear operator $\Lambda:\ell^2(\pi_Q) \to \ell^2(\pi)$ ($\pi_Q$ being a reference measure for $Q$) with bounded inverse such that
	\begin{equation}
	P \Lambda = \Lambda Q.
	\end{equation}
	We also write $\widehat{\Lambda}$ to be the adjoint operator of $\Lambda$. When needed we may write  $P \stackrel{\Lambda}{\sim} Q$ to specify the intertwining or the link kernel $\Lambda$. Note that $\sim$ is an equivalence relationship on the set of transition kernels $\M$.
\end{definition}

\begin{rk}
	In the discrete-time setting, for $n \in \mathbb{N}$, if $P \stackrel{\Lambda}{\sim} Q$, then $P^n \stackrel{\Lambda}{\sim} Q^n$.
\end{rk}

\begin{rk}\label{rk:cont}
	Note that this definition carries over when we study similarity on the level of infinitesimal generator in the continuous-time setting. For example, we write $L \stackrel{\Lambda}{\sim} G$ if $L$ (resp.~$G$) is the infinitesimal generator associated with the continuous-time Markov semigroup $(P_t)_{t \geq 0}$ (resp.~$(Q_t)_{t \geq 0}$). It follows easily that if $L \stackrel{\Lambda}{\sim} G$ then $P_t \stackrel{\Lambda}{\sim} Q_t$ for $t \geq 0$.
\end{rk}

\begin{rk}
	We compare our definition of similarity with other notions of intertwining in the literature. In \cite{Miclo2016}, both the link kernel $\Lambda$ and $Q$ are assumed to be Markovian, while in \cite{ChafaiJoulin13,CD16}, both $P$ and $Q$ are assumed to be birth-death processes. In \cite{FD}, the authors construct the strong stationary duality theory for general Markov chains. However, in Definition \ref{def:sim} we do not require $\Lambda$ to be a Markov operator, and $P,Q$ are general Markov operators instead of birth-death processes.
\end{rk}

The $\S$ class is now defined as the similarity orbit in $\M$ consisting of all Markov transition kernels that are similar to a normal transition kernel on $\mathcal{X}$. Note that reversible Markov kernels are normal operators in $\ell^2(\pi)$. From now on, we write $\mathcal{N}$ to be the set of normal transition kernels $Q$ on $\mathcal{X}$, that is, $Q\widehat{Q} = \widehat{Q}Q$ in $\ell^2(\pi_Q)$ where  $\: \widehat{}~$ denotes throughout the corresponding object for  the time-reversal process.

\begin{definition}[The $\S$ class]\label{def:Sclass}
	Suppose that $Q \in \mathcal{N}$. The similarity orbit of $Q$ (in $\Sf$) is
	\[ \S(Q) = \{P \in \mathcal{M};\:  P \sim Q \}, \]
	and the $\S$ class is the union over all possible orbits
	\[ \S = \bigcup_{Q \in \mathcal{N}} \S(Q). \]
	We point out that according to 
	\cite{Wermer54}, the class $\mathcal{S}$ is also characterized as the class of Markov chain whose transition kernel is a spectral scalar-type operator in the sense of \cite[Section $3$]{Dunford54}, see also \cite[Page $1938$, Definition $1$]{DS71}. As we will see in Section \ref{sec:st}, this characterization will be crucial in proving many of our later results.
Note that we could also study a wider class of transition kernel $\mathcal{S}^{'}$ in which $Q$ is not necessarily a Markov operator. However, we intend to focus our investigation on the class $\mathcal{S}$ in this paper as it is the appropriate setting to extend various substantial results that have been developed  for reversible chains.	

\end{definition}

We now summarize the major contributions of this work in the analysis of general Markov chains which also serve as an outline of the paper. In Section \ref{sec:st}, we begin by showing how the concept of similarity orbit is natural for developing the spectral decomposition of non-reversible Markov operators in the class $\S$. Indeed, each of its element admits a spectral representation with respect to non-self-adjoint resolution of identity as introduced by \cite{Dunford54}, see also \cite{DS71}. We also remark on the growing interest for non-self-adjoint operators with real spectrum that arise in the study of pseudo-hermitian quantum mechanics, see e.g. \cite{IT14} and the references therein. As by-product, one can develop a functional calculus for this class as for normal operators. Moreover, we obtain, under mild conditions, an eigenvalues expansion expressed in terms of Riesz basis, a notion that generalizes orthogonal basis and was introduced in non-harmonic analysis, see \cite{Y01}. Another intriguing aspect of the similarity orbit analysis is that in the continuous-time setting with $L \in \S(G)$ (see Remark \ref{rk:cont} above), where $G$ is the generator of a normal Markov chain, then both the heat kernel $(e^{tL})_{t \geq 0}$ and $(e^{tG})_{t \geq 0}$ share the same \textit{eigentime} identity, offering new examples and insights to the sequence of work by \cite{AF14,CuiMao10} and \cite{Miclo15}. Added to the above, we obtain a two-phase refinement for the convergence rate of the Markov kernels in the class $\S$ measured in the Hilbert space topology or in total variation distance: recall that in the normal case the rate of convergence in the Hilbert space topology is given by \textit{exactly} the second largest eigenvalue in modulus; for class $\S$ however, in small time we adapt the singular value upper bound of \cite{Fill91}, while for large time, the decay rate is the second largest eigenvalue in modulus modulo a constant which is the condition number of the link kernel $\Lambda$. This offers an original spectral explanation of the hypocoercivity phenomenon that has been observed and studied intensively in the PDE literature, see for instance \cite{V09}.
All these first consequences of the spectral representation are stated and proved in Section \ref{sec:st}. In view of the tractability and the fascinating properties that the class $\mathcal{S}$ possesses, it will be very interesting to characterize this class in terms of the one-step transition probabilities of $P \in \mathcal{S}$. Although fundamental,  this issue seems to be  very challenging. However, we manage to identify a set of sufficient conditions that defines what we call the generalized monotonicity condition class $\GMc$, such that the time-reversal $\widehat{P}$ intertwines with a birth-death chain in Section \ref{sec:sepcut}. This $\GMc$ class rests on the assumption of stochastic monotonicity in which $\Lambda$ is the so-called Siegmund kernel. This readily generalizes the $\Mc$ class introduced by \cite{Choi-Patie} in the context of skip-free chains. Note that the notion of stochastic monotonicity is studied by \cite{Sieg76} and \cite{CS85} and intertwining between stochastic monotone birth-death chains, which are reversible chains, has been previously investigated in detail by \cite{FD,HM11} and \cite{JK14}. Relying on the spectral decomposition as well as the fastest strong stationary time result of general chains obtained by \cite{Fill}, we study the separation cutoff phenomenon and demonstrate that the famous ``spectral gap times mixing time" conjecture as well as the proof in \cite{DSC06} carries over to the subclass $\GMc^+ \subset \GMc$ in Section \ref{sec:sepcut}. Next, building upon the concept of the non-self-adjoint spectral measure and the Laplace transform cutoff criteria proposed in \cite{CSC10} and further elaborated in \cite{CHS17}, we illustrate that the usual ${\rm{L}}^2$-cutoff criteria for reversible chains generalizes to the class $\S$ in Section \ref{sec:l2cut}.

Second, in Section \ref{sec:nonasympest}, we would like to estimate integral functionals of the type
$$\Gamma_T(f) = \int_0^T f(X_t)\, dt, \quad T \geq 0,$$
where $T$ is a fixed time and $f$ is a function such that the integral $\Gamma_T(f)$ is well-defined, by the Riemann-sum estimator given by, for $n \in \mathbb{N}$,
$$\hat{\Gamma}_{T,n}(f) = \sum_{k=1}^n f(X_{(k-1)\Delta_n}) \Delta_n,$$
where we observe $(X_t)_{t \in [0,T]}$ at discrete epochs $t = (k-1)\Delta_n$ with $k \in \llbracket n \rrbracket := \{1,\ldots,n\}$ and $\Delta_n = T/n$. This work is motivated by the recent work of \cite{AC16}, in which they studied the same problem with the outstanding assumption that the infinitesimal generator of the Markov process $(X_t)_{t \geq 0}$ is a normal operator to yield interesting results on the estimator error bound by spectral theory. We demonstrate that a number of their results can be readily generalized to the class $\S$ on the infinitesimal generator level.

Finally, in Section \ref{sec:simorbit}, we examine a particular similarity orbit of reversible Markov chains that we call the pure birth orbit. More precisely, suppose that we start with a reversible generator $G$ such that $G \stackrel{\Lambda}{\sim} L$ , where $L$ is the generator of a contraction yet possibly non-Markovian semigroup $(e^{tL})_{t \geq 0}$, we would like to investigate various properties of $L$ with $\Lambda$ being a pure birth kernel. This idea is powerful enough to allow us to generate completely new Markov or contraction kernel from known ones in which we have precise control and exact expressions on the stationary distribution, eigenfunctions and the speed of convergence. In particular, we perform an in-depth study on the pure birth variants of a constant rate birth-death model.

\section{Spectral theory of the class $\S$ and its convergence rate to equilibrium}  \label{sec:st}
In this Section, we develop an original methodology to obtain the spectral decomposition in the Hilbert space of the transition operator of Markov chains that  belong to the class $\S$, a subclass of $\M$ which is defined in Definition \ref{def:Sclass}. We write $\norm{\cdot}_{op}$ to be the operator norm, i.e.~$\norm{P}_{op}=\sup_{||f||_{\pi}=1}||Pf||_{\pi}$, and $\llbracket a,b \rrbracket := \{a,a+1,\ldots,b-1,b\}$ for any $a \leq b \in \mathbb{Z}$. We proceed by recalling that $P$ has a time-reversal $\widehat{P}$, that is, for $x,y \in \mathcal{X}$,
$$\pi(x) \widehat{P}(x,y) = \pi(y) P(y,x),$$
where $\pi$ is a reference measure for $P$. We equip the Hilbert space $\ell^2(\pi)$ with the usual inner product $\langle\cdot,\cdot\rangle_{\pi}$ defined by
$$\langle f,g \rangle_{\pi} = \sum_{x \in \mathcal{X}} f(x) \overline{g}(x) \pi(x), \quad f,g \in \ell^{2}(\pi),$$
where $\overline{g}$ is the complex conjugate of $g$. A spectral measure (or resolution of identity) in the sense of \cite[Section $3$]{Dunford54} and \cite[Page $1929$ Definition $1$]{DS71} of a Hilbert space $\mathcal{H}$ on $\mathbb{C}$ is a family of bounded operators $\mathcal{E} = \{E_B; B \in \mathcal{B}(\mathbb{C})\}$, where $\mathcal{B}(\mathbb{C})$ is the Borel algebra on $\mathbb{C}$, satisfying the following:

\begin{enumerate}
	\item $E_{\emptyset} = 0, E_{\mathbb{C}} = I$.
	\item For all $A,B \in \mathcal{B}(\mathbb{C})$,
	$$E_{A \cap B} = E_A E_B,$$
	while for disjoint $A,B$,
	$$E_{A \cup B} = E_A + E_B.$$
	\item There exists a constant $C>0$ such that $\norm{E_B}_{op} \leq C$ for all $B \in \mathcal{B}(\mathbb{C})$.
\end{enumerate}
For normal operator $Q \in \mathcal{N}$, its resolution of identity $\mathcal{E}$ is self-adjoint and hence $\mathcal{E}$ is a self-adjoint orthogonal projection. We also denote $E_B^*$ to be the adjoint of $E_B$. Recall that by the spectral theorem for normal operators the spectral resolution of $Q$ is
$$Q = \int_{\sigma(Q)} \lambda \,d E_{\lambda},$$
where $\sigma(Q)$ is the spectrum of $Q$. More generally, for $M \in \mathcal{M}$, we write $\sigma(M)$ (resp.~$\sigma_c(M)$, $\sigma_p(M)$, $\sigma_r(M)$) to be the spectrum (resp.~continuous spectrum, point spectrum, residual spectrum) of $M$.
We proceed to recall the notion of Riesz basis, which will be useful when we derive the spectral decomposition for compact $P \in \S$ in our main result Theorem \ref{thm:mcins} below. A basis $(f_k)$ of a Hilbert space $\mathcal{H}$ is a Riesz basis if it is obtained from an orthonormal basis $(e_k)$ under a bounded invertible operator $T$, that is, $T e_k = f_k$ for all $k$. It can be shown, see e.g.~\cite[Theorem $9$]{Y01}, that the sequence $(f_k)$ forms a Riesz basis if and only if $(f_k)$ is complete in $\mathcal{H}$ and there exist positive constants $A,B$ such that for arbitrary $n \in \N$ and scalars $c_1,\ldots,c_n$, we have
\begin{equation}\label{eq:Rieszb}
A \sum_{k=1}^n |c_k|^2 \leq \norm{\sum_{k=1}^n c_k f_k}^2 \leq B \sum_{k=1}^n |c_k|^2.
\end{equation}
If $(g_k)$ is a biorthogonal sequence to $(f_k)$, that is, $\langle f_k, g_m \rangle_{\pi} = \delta_{k,m}$, $k,m \in \N$ and $\delta_{k,m}$ is the Kronecker symbol, then $(g_k)$ also forms a Riesz basis. We are now ready to state the main result of this paper in the following, and the proof can be found in Section \ref{subsec:proof1}.

\begin{theorem}\label{thm:mcins}
	Assume that $P \in \S$ with   $P \stackrel{\Lambda}{\sim} Q ~\in \mathcal{N}$. Then the following holds.
	\begin{enumerate}[(a)]
		\item\label{it:resolution} Denote the self-adjoint spectral measure of $Q$ by $\mathcal{E} = \{E_B;~B \in \mathcal{B}(\mathbb{C})\}$, then $\{F_{B} := \Lambda E_{B} \Lambda^{-1};~B \in \mathcal{B}(\mathbb{C})\}$ defines a spectral measure and $P$ is a spectral scalar-type operator 
		with spectral resolution given by
		\begin{align*}
			P &= \int_{\sigma(P)} \lambda \,d F_{\lambda}, \\
			\widehat{P} &= \int_{\sigma(\widehat{P})} \lambda \,d F_{\lambda}^*.
		\end{align*}
		Note that
		$$\sigma(P) = \sigma(Q), \sigma(P) = \overline{\sigma(\widehat{P})}, \sigma_c(P) = \sigma_c(Q), \sigma_p(P) = \sigma_p(Q), \sigma_r(P) = \sigma_r(Q),$$
		and the multiplicity of each eigenvalue in $\sigma_p(P)$ is the same as that of $\sigma_p(Q)$. For analytic and single valued function $f$ on $\sigma(P)$, we have
		\begin{align*}
			f(P) &= \int_{\sigma(P)} f(\lambda) \,d F_{\lambda}.
		\end{align*}
		In particular, if $P$ is compact on $\mathcal{X}$ with distinct eigenvalues then for any $f \in \ell^{2}(\pi)$ and $n \in \mathbb{N}$, \[ P^n f = \sum_{k \in \mathcal{X}} \lambda_k^n \langle f, f_k^* \rangle_{\pi} f_k,\]
		where the set $(f_k)$ are eigenfunctions of $P$ associated to the eigenvalues $(\lambda_k)$ and form a Riesz basis of $\ell^2(\pi)$, and the set $(f_k^*)$ is the unique Riesz basis biorthogonal to $(f_k)$. For any $x,y \in \mathcal{X}$ and $n \in \mathbb{N}$, the spectral expansion of $P$ is given by
		$$ P^n(x,y) = \sum_{k \in \mathcal{X}} \lambda_k^n f_k(x) f_k^*(y)\pi(y).$$
		\item \label{it:sbd} $P \stackrel{\Lambda}{\sim} Q$ if and only if $\widehat{Q} \stackrel{\widehat{\Lambda}}{\sim} \widehat{P}$.
		\item\label{it:unitary} Suppose that $\Lambda$ is an unitary operator, that is, $\Lambda^{-1} = \widehat{\Lambda}$. Then $P$ is a normal (resp.~self-adjoint) operator in $\ell^2(\pi)$ if and only if $Q$ is a normal (resp.~self-adjoint) operator in $\ell^2(\pi_Q)$.
		\item(Lattice isomorphism)\label{it:lattice} Suppose that $\mathcal{X}$ is a finite state space. $\Lambda$ is an invertible Markov kernel on $\mathcal{X}$ with $\Lambda^{-1}$ having non-negative entries if and only if $\Lambda \in \mathcal{P}$, the set of  permutation kernels. We recall that  $\Lambda \in \mathcal{P}$ if $\Lambda = \Lambda_{\sigma} := (\1_{y = \sigma(x)})_{x,y \in \mathcal{X}}$ with $\sigma: \mathcal{X} \mapsto \mathcal{X}$ being a permutation of the state space, and note that $\Lambda_{\sigma}$ is an unitary Markov kernel. Moreover, for any $Q \in \mathcal{M}$, the permutation orbit $\mathcal{S}_{\mathcal{P}}(Q)$ of $Q$ is given by $\mathcal{S}_{\mathcal{P}}(Q)=\{ P \in \overline{\mathcal{M}}; P \Lambda=\Lambda Q, \Lambda \in \mathcal{P}\} \subset \mathcal M$, where $\overline{\mathcal{M}}$ is the set of square matrices on $\mathcal{X}$.
		\item \label{it:eig} Suppose that $\mathcal{X}$ is a finite state space and $Q$ is the transition kernel of an irreducible birth-death process, then $P \stackrel{\Lambda}{\sim} Q$ if and only if $P$ has real and distinct eigenvalues.
	\end{enumerate}
\end{theorem}

\begin{rk}
	As suggested by item \ref{it:unitary}, we can generate new non-normal examples via non-unitary link from known normal Markov chains such as birth-death processes. In Section \ref{sec:simorbit}, we investigate a particular non-unitary orbit that we call the pure birth orbit.
\end{rk}

\begin{rk}
	The result in Theorem \ref{thm:mcins}\ref{it:lattice} has also been obtained by Miclo using a different proof, see Lemma $12$ in \cite{miclo2015markov}.
\end{rk}

\begin{rk}
	The key to Theorem \ref{thm:mcins}\ref{it:eig} lies on the simplicity of the spectrum of $Q$. In the context of non-negative Jacobi matrices, the inverse eigenvalue problem has been studied by \cite[Theorem $4$]{FM79}.
\end{rk}

\begin{rk}\label{rk:cont2}
	Theorem \ref{thm:mcins} can be generalized easily to the continuous-time setting, see also Remark \ref{rk:cont}. Indeed, suppose that $L \in \mathcal{S}(G)$, where $G$ is a normal generator with spectral measure $\mathcal{E} = \{E_B;~B \in \mathcal{B}(\mathbb{C})\}$, then for $t \geq 0$,
	$$P_t = \int_{\sigma(L)} e^{t\lambda} dF_{\lambda},$$
	where $\{F_{B} := \Lambda E_{B} \Lambda^{-1};~B \in \mathcal{B}(\mathbb{C})\}.$
\end{rk}	


As a first application of the spectral decomposition stated in Theorem \ref{thm:mcins}, we derive accurate information regarding the speed of convergence to stationarity for ergodic chains in $\S$ in both the Hilbert space topology and in total variation distance.   There has been a rich literature devoted to the study of convergence to equilibrium for non-reversible chains by means of reversibilizations, see e.g.~\cite{Fill91, AF14, LPW09,MT06} and the references therein. Our approach reveals a natural extension to the non-reversible case of the classical spectral gap that appears in the study of reversible chains. To state our result we now fix some notations. We denote the second largest eigenvalue in modulus (SLEM) or the spectral radius of $P$ in the Hilbert space $$\ell^2_0(\pi) = \{ f \in \ell^2(\pi); \:~\langle f,\mathbf{1} \rangle_{\pi} = 0 \},$$  by $$\lambda_{*} = \lambda_{*}(P) = \sup\{|\lambda_i|;~ \lambda_i \in \sigma(P),\lambda_i \neq 1\},$$ then the \textit{absolute spectral gap} is $\gamma_* = 1 - \lambda_*$. For any two probability measures $\mu, \nu$ on $\mathcal{X}$, the total variation distance between $\mu$ and $\nu$ is given by
$$|| \mu - \nu ||_{TV} = \dfrac{1}{2} \sum_{x \in \mathcal{X}} |\mu(x) - \nu(x)|.$$
For $n \in \mathbb{N}$, the total variation distance from stationarity of $X$ is
$$d(n) = \max_{x \in \mathcal{X}} || \delta_x P^n -  \pi ||_{TV}.$$
For $g \in \ell^2(\pi)$, the mean of $g$ with respect to $\pi$ can be written as $\E_{\pi}(g) = \langle g,\mathbf{1} \rangle_{\pi}$. Similarly, the variance of $g$ with respect to $\pi$ is $\mathrm{Var}_{\pi}(g) = \langle g,g \rangle_{\pi} - \E_{\pi}^2(g) $.
Finally, we recall that Fill in \cite[Theorem $2.1$]{Fill91} obtained in the finite state space case the following  bound valids for all  $n \in \mathbb{N}_0$
\begin{align}\label{eq:tvfill}
d(n) \leq \dfrac{\sigma^n_*(P)}{2} \sqrt{\dfrac{1-\pi_{min}}{\pi_{min}}},
\end{align}
where $\pi_{min} = \min\limits_{x \in \mathcal{X}} \pi(x)$ and $\sigma_*(P) = \sqrt{\lambda_*(P \widehat{P})}$ is the second largest singular value of $P$. We obtain the following refinement for Markov chains in the class $\S$. The proof is deferred to Section \ref{subsec:proof2}.

\begin{corollary}\label{cor:spectralexp}
	Let $P \in \S$ with invariant distribution $\pi$, that is, $\pi P = \pi$, and assume that $P$ is compact.
	\begin{enumerate}
		\item \label{it:varbd} For any $n \in \mathbb{N}_0$,  we have
		\begin{align} \label{eq:hyp}
		\lambda_*^n\leq \norm{P^n - \pi}_{\ell^2(\pi) \to \ell^2(\pi)} \leq  \min\left(\sigma^{n}_*(P),{\kappa(\Lambda)}  \lambda_*^n \right) = \sigma^{n}_*(P)\1_{\{n<n^*\}} +\kappa(\Lambda)  \lambda_*^n\1_{\{n\geq n^*\}},
		\end{align}
		where $n^* = \lceil \frac{\ln \kappa(\Lambda)}{\ln \sigma_*(P)-\ln\lambda_*} \rceil$ and $\kappa(\Lambda)= \norm{\Lambda}_{\ell^2(\pi_Q) \to \ell^2(\pi)} \: \norm{\Lambda^{-1}}_{\ell^2(\pi) \to \ell^2(\pi_Q)} \geq 1$ is the condition number of $\Lambda$.
		When $\mathcal{X}$ is a finite state space, a sufficient condition for which $\lambda_* < \sigma_*(P)$ is given by $\max_{i \in \mathcal{X}} P(i,i) > \lambda_*$. In such case, for $n$ large enough, the convergence rate $\lambda_*$ given \eqref{eq:hyp} is strictly better than the reversibilization rate $\sigma_*(P)$.
		\item \label{it:totalvar} Suppose that $\mathcal{X}$ is a finite state space. For any $n \in \mathbb{N}_0$,
		\begin{align*}
		d(n) \leq \dfrac{\min\left(\sigma^{n}_*(P),{\kappa(\Lambda)}  \lambda_*^n \right)}{2} \sqrt{\dfrac{1-\pi_{min}}{\pi_{min}}},
		\end{align*}
		where $\lambda_{*}\leq \sigma_*(P)$.
	\end{enumerate}
\end{corollary}
\begin{rk}
	Recall that when $P$ is reversible and compact then the sequence of eigenfunctions is orthonormal and  thus an application of the Parseval identity yields the well-known result (see e.g. \cite[Section $4.3$]{CSC10}) $\norm{P^n - \pi}_{\ell^2(\pi) \to \ell^2(\pi)} =  \lambda_*^n$ and $\kappa(\Lambda) = 1$ which  is a specific instance of  item \eqref{it:varbd}. 
\end{rk}
\begin{rk}
	We also recall the discrete analogue of the notion of hypocoercivity introduced in \cite{V09}, i.e.~there exists a constant $C < \infty$ and $\rho \in (0,1)$ such that, for all $ n \in \mathbb{N}$,
	$$\norm{P^n - \pi}_{\ell^2(\pi) \to \ell^2(\pi)} \leq C \rho^n.$$
	Note that, in general, these constants are not known explicitly. We observe that the upper bound in \eqref{eq:hyp} reveals that the ergodic chains in $\S$ satisfy this hypocoercivity phenomena. More interestingly, our approach based on the similarity concept enables us to get on the one hand an explicit and on the other hand a spectral interpretation of this rate of convergence. Indeed, it can be understood as a modified spectral gap  where the perturbation  from the classical spectral gap is given by the condition number $\kappa(\Lambda)$ which can be interpreted as a measure of deviation from symmetry. In this vein, we  mention the recent work \cite{Patie-Savov} where a similar spectral interpretation of the hypocoercivity phenomena is given for a class of non-self-adjoint Markov semigroups.
\end{rk}

\begin{rk}
	Here we provide an alternative expression for the upper bound of \eqref{eq:hyp}. Let $\ell^2(\mathcal{X})$ be the space of square summable functions on $\mathcal{X}$ equipped with the standard inner product $\langle\cdot,\cdot\rangle$. Suppose that $P \in \mathcal{S}(Q)$ with $QU = UD$, where $D$ is a diagonal matrix and $U$ is an isometry from $\ell^2(\mathcal{X})$ to $\ell^2(\pi_Q)$. Then $B = \Lambda U$ is an eigenbasis for $P$ with $\kappa(B) = \kappa(\Lambda)$, where $\kappa(B) := \norm{B}_{\ell^2(\mathcal{X}) \to \ell^2(\pi)} \norm{B^{-1}}_{\ell^2(\pi) \to \ell^2(\mathcal{X})}$ is the condition number of $B$, so the upper bound in \eqref{eq:hyp} can be written as
	$$\norm{P^n - \pi}_{\ell^2(\pi) \to \ell^2(\pi)} \leq  \kappa(B)  \lambda_*^n.$$
\end{rk}

As a second application of Theorem \ref{thm:mcins}, we first recall the celebrated eigentime identity studied by \cite{AF14,CuiMao10} and \cite{Miclo15}: suppose that we sample two points $x$ and $y$ randomly from the stationary distribution of the chain and calculate the expected hitting time from $x$ to $y$, the expected value of this procedure is the sum of the inverse of the non-zero (and negative of the) eigenvalues of the generator. Since similarity preserves the eigenvalues (see Theorem \ref{thm:mcins} item \ref{it:resolution}), we can easily see that both $P$ and $Q$ share the same eigentime identity:

\begin{corollary}[Eigentime identity]\label{cor:eigentime}
	Suppose that $\mathcal{X}$ is a finite state space and $(Q_t)_{t \geq 0}$ (resp.~$(P_t)_{t \geq 0}$) has generator $G$ (resp.~$L$) associated with the ergodic Markov chain $(X_t)_{t \geq 0}$ (resp.~$(Y_t)_{t \geq 0}$). If $L \in \mathcal{S}(G)$ with $G$ being a normal generator and common eigenvalues $(-\lambda_i)_{i \in \llbracket |\mathcal{X}| \rrbracket}$, then $(P_t)_{t \geq 0}$ and $(Q_t)_{t \geq 0}$ share the same eigentime identity. That is, denote $\tau_y^Q := \inf\{t \geq 0; X_t = y\}$~(resp.~$\tau_y^P := \inf\{t \geq 0; Y_t = y\}$), then
	$$\sum_{x,y \in \mathcal{X}} \E_x(\tau_y^Q) \pi_Q(x) \pi_Q(y) = \sum_{x,y \in \mathcal{X}} \E_x(\tau_y^P) \pi(x) \pi(y) = \sum_{i=1, \lambda_i \neq 0}^{|\mathcal{X}|} \dfrac{1}{\lambda_i}.$$
\end{corollary}

\subsection{Proof of Theorem \ref{thm:mcins}}\label{subsec:proof1}
	We first show the item \ref{it:resolution}. Since $\mathcal{E}$ is a spectral measure, it follows easily that $\{F_B = \Lambda E_B \Lambda^{-1};~B \in \mathcal{B}(\mathcal{C})\}$ is a spectral measure. The fact that the spectrum coincides and
	$$\sigma(P) = \sigma(Q), \sigma(P) = \overline{\sigma(\widehat{P})}, \sigma_c(P) = \sigma_c(Q), \sigma_p(P) = \sigma_p(Q), \sigma_r(P) = \sigma_r(Q),$$ follows from Proposition $3.9$ in \cite{AT14}. Define $\overline{P} := \int_{\sigma(P)} \lambda \,d F_{\lambda}.$ We have
	$$\overline{P} = \int_{\sigma(P)} \lambda \,d (\Lambda E_{\lambda} \Lambda^{-1}) = \Lambda \left(\int_{\sigma(Q)} \lambda \,d E_{\lambda}\right) \Lambda^{-1} = \Lambda Q \Lambda^{-1} = P,$$
	so the desired spectral resolution of $P$ follows, thus it is a spectral scalar-type operator. The spectral resolution of $\widehat{P}$ follows from that of $P$. The functional calculus of $P$ follows immediately from that of spectral scalar-type operator, see e.g. Theorem $1$ in Chapter $XV.5$, Page $1941$ of \cite{DS71}. We proceed to handle the case when $P$ is compact.
	Denote $(g_k)$ to be the (orthogonal) eigenfunctions of the normal transition kernel $Q$. Since $f_k = \Lambda g_k$ and $\Lambda$ is bounded, $(f_k)$ is complete as $(g_k)$ is a basis. As $\Lambda$ is bounded from above and below, for any $n \in \N$ and arbitrary sequence $(c_k)_{k=1}^{n}$, we have
	\begin{align*}
	A \sum_{k=1}^n |c_k|^2 \leq \norm{\sum_{k=1}^n c_k f_k}^2_{\pi} = \norm{\Lambda \sum_{k=1}^n c_k g_k}^2_{\pi} \leq B \sum_{k=1}^n |c_k|^2,
	\end{align*}
	where we can take $A = \norm{\Lambda^{-1}}^{-2}$ and $B = \norm{\Lambda}^2$, so \eqref{eq:Rieszb} is satisfied. It follows from \cite[Theorem $9$]{Y01} that there exists the sequence $(f_k^*)$ being the unique Riesz basis biorthogonal to $(f_k)$, and, any $f \in \ell^2(\pi)$ can be written as
	$$f = \sum_{k \in \mathcal{X}} c_k f_k,$$
	where $c_k = \langle f , f_k^* \rangle_{\pi}$. Desired result follows by applying $P^n$ to $f$ and using $P^n f_k = \lambda_k^n f_k$. In particular, if we take $f = \delta_y$, the Dirac mass at $y$, and evaluate the resulting expression at $x$, we obtain the spectral expansion of $P$. Next, we show item \ref{it:sbd}. If $P \stackrel{\Lambda}{\sim} Q$, then for $f \in \ell^2(\pi_{Q})$ and $g \in \ell^2(\pi)$,
	$$\langle f, \widehat{\Lambda} \widehat{P} g \rangle_{\pi_{Q}} = \langle P \Lambda f, g \rangle_{\pi} = \langle \Lambda Q f, g \rangle_{\pi} = \langle f, \widehat{Q} \widehat{\Lambda} g \rangle_{\pi_{Q}},$$
	which shows that $\widehat{Q} \stackrel{\widehat{\Lambda}}{\sim} \widehat{P}$. The opposite direction can be shown similarly. For item \ref{it:unitary}. Since $\Lambda$ is unitary, the spectral measures of $P$ and $Q$ are related by $F_B = \Lambda E_B \widehat{\Lambda}$, so $F_B$ is self-adjoint if and only if $E_B$ is self-adjoint, which implies that $P$ is normal if and only if $Q$ is normal. If $Q$ is self-adjoint, then item \ref{it:sbd} yields $P \stackrel{\Lambda}{\sim} Q$ if and only if $Q \stackrel{\Lambda^{-1}}{\sim} \widehat{P}$, which implies that $\widehat{P} = P$ in $\ell^2(\pi)$. The opposite direction can be shown similarly. Next, we show item \ref{it:lattice}. If $\Lambda$ is a permutation link, then it is trivial to see that $\Lambda$ is an invertible Markov kernel. For the opposite direction, it is known (see e.g. \cite[Section $3$]{BP74}) that $\Lambda = D \Lambda_{\sigma}$, where $D$ is a diagonal matrix. We then have
	$\mathbf{1} = \Lambda \mathbf{1} = D \Lambda_{\sigma} \mathbf{1} = D \mathbf{1},$
	which gives $D = I$, and hence $\Lambda = \Lambda_{\sigma}$. Let now $Q \in \mathcal{M}$ and $P \in \mathcal{S}_{\mathcal{P}}(Q)$, then since  $P=\Lambda Q \Lambda^{-1}$ with $\Lambda,\Lambda^{-1} \in \mathcal{P}$, we deduce readily that $P \in \mathcal{M}$.
	Finally, to show item \ref{it:eig}, if $P \stackrel{\Lambda}{\sim} Q$, then $P$ has real and distinct eigenvalues since $Q$ has real and distinct eigenvalues. Conversely, if $P$ has real and distinct eigenvalues, $P$ is diagonalizable, so there exists an invertible $\Lambda$ such that
	$$P = \Lambda D \Lambda^{-1}.$$
	where $D$ is the diagonal matrix storing the eigenvalues of $P$. Given the spectral data $D$, by inverse spectral theorem, see e.g.~\cite[Section 5.8]{DM76}, one can always construct an ergodic Markov chain with transition matrix $Q$ such that
	$$Q = VDV^{-1}.$$

\subsection{Proof of Corollary \ref{cor:spectralexp}}\label{subsec:proof2}
	We first show the upper bound in item \eqref{it:varbd}. Define the synthesis operator $T^* : \ell^2 \to \ell^2(\pi)$ by $\alpha = (\alpha_i) \mapsto T^*(\alpha) = \sum_{i \in \mathcal{X}} \alpha_i f_i$, where $(f_i)$ are the eigenfunctions of $P$ and $(f_i^*)$ are the unique biorthogonal basis of $(f_i)$ as in Theorem \ref{thm:mcins}. For $i \in \mathcal{X}$, we take $\alpha_i = \alpha_i(n) = \lambda_i^n \langle g,f_i^* \rangle_{\pi}$, and denote $(q_i)$ to be the orthonormal eigenfunctions of $Q ~\in \mathcal{N}$, where $f_i = \Lambda q_i$. Note that $||T^*||_{op} \leq ||\Lambda||_{\ell^2(\pi_Q) \to \ell^2(\pi)} < \infty$, since
	$$||T^*(\alpha)||_{ \ell^2 \to \ell^2(\pi)} = \norm{\sum_{i \in \mathcal{X}} \alpha_i \Lambda q_i}_{ \ell^2 \to \ell^2(\pi)} \leq \norm{\Lambda}_{\ell^2(\pi_Q) \to \ell^2(\pi)} \norm{\sum_{i \in \mathcal{X}} \alpha_i q_i}_{\pi_Q} \leq ||\Lambda||_{\ell^2(\pi_Q) \to \ell^2(\pi)} ||\alpha||_{\ell^2}.$$
	For $g \in \ell^2_0(\pi)$, we also have
	\begin{align*}
	\sum_{i \in \mathcal{X}} |\langle g,f_i^* \rangle_{\pi}|^2 &= \sum_{i \in \mathcal{X}} |\langle g, (\Lambda^{*})^{-1}q_i \rangle_{\pi}|^2
	= \sum_{i \in \mathcal{X}} |\langle \Lambda^{-1}g, q_i \rangle_{\pi_Q}|^2
	= ||\Lambda^{-1}g||^2_{\pi_Q}
	\leq ||\Lambda^{-1}||^2_{\ell^2(\pi) \to \ell^2(\pi_Q)} ||g||^2_{\pi},
	\end{align*}
	where the third equality follows from Parseval's identity, which leads to
	\begin{align}\label{eq:l2}
	|| P^ng - \pi g||_{\pi}^2 = ||T^*(\alpha)||^2_{\ell^2 \to \ell^2(\pi)} \leq ||\Lambda||_{\ell^2(\pi_Q) \to \ell^2(\pi)}^2 ||\alpha||_{l^2}^2 \leq ||\Lambda||_{\ell^2(\pi_Q) \to \ell^2(\pi)}^2 ||\Lambda^{-1}||_{\ell^2(\pi) \to \ell^2(\pi_Q)}^2 \lambda_*^{2n}  ||g||^2_{\pi}.
	\end{align}
	The desired upper bound follows from \eqref{eq:l2} and  \begin{equation*}
	\norm{P^n - \pi}_{\ell^2(\pi) \to \ell^2(\pi)} \leq \lambda_*(\widehat{P}P)^{n/2} = \lambda_*(P\widehat{P})^{n/2},
	\end{equation*}
	see e.g.~\cite{Fill91}. The lower bound in \eqref{it:varbd} follows readily from the well-known result that the $n^{th}$ power of the spectral radius $\lambda_*^n$ is less than or equal to the norm of $P^n$ on the reduced space $\ell^2_0(\pi)$.
	For the sufficient condition in item \eqref{it:varbd}, that is, $\max_{i \in \mathcal{X}} P(i,i) > \lambda_*$ implies $\lambda_* < \sigma_*(P)$, it is a straightforward consequence of the Sing-Thompson Theorem, see \cite{Thompson77}.
	Next, we show item \eqref{it:totalvar}. Using \eqref{eq:l2}, we get
	\begin{align}\label{eq:vareig}
	\mathrm{Var}_{\pi} \left({\widehat{P}}^n g \right) &\leq \kappa(\widehat{\Lambda})^2 \lambda_*^{2n} \mathrm{Var}_{\pi}(g) = \kappa(\Lambda)^2\lambda_*^{2n} \mathrm{Var}_{\pi}(g), \quad n \in \mathbb{N}_0,
	\end{align}
	where we used the obvious identity $\kappa(\Lambda) = \kappa(\widehat{\Lambda})$ in the equality. This leads to
	\begin{align*}
	|| \delta_x P^n -  \pi ||_{TV}^2 &= \dfrac{1}{4} \E_{\pi}^2 \left| \dfrac{\delta_x P^n}{\pi} - 1 \right|
	\leq \dfrac{1}{4} \mathrm{Var}_{\pi} \left( \dfrac{\delta_x P^n}{\pi} \right)
	= \dfrac{1}{4} \mathrm{Var}_{\pi} \left({\widehat{P}}^n \frac{\delta_x}{\pi} \right)
	\leq \dfrac{1}{4} \kappa(\Lambda)^2 \lambda_{*}^{2n} \mathrm{Var}_{\pi} \left( \frac{\delta_x}{\pi} \right) \\
	&= \dfrac{1}{4} \kappa(\Lambda)^2 \lambda_{*}^{2n} \dfrac{1 - \pi(x)}{\pi(x)} \leq \dfrac{1}{4} \kappa(\Lambda)^2 \lambda_{*}^{2n} \dfrac{1 - \pi_{min}}{\pi_{min}},
	\end{align*}
	where the first inequality follows from Cauchy-Schwartz
	inequality. The proof is completed  by combining the above bound with \eqref{eq:tvfill}.

\section{The $\GMc$ class and separation cutoff}\label{sec:sepcut}

As Theorem \ref{thm:mcins} suggests, the class $\mathcal{S}$ is highly tractable and enjoys a number of attractive properties. It will therefore be very interesting to characterize this class in terms of the one-step transition probabilities of $P$, which is a fundamental yet challenging issue. However, we manage to identify a set of sufficient conditions that we call the generalized monotonicity condition class $\GMc$, generalizing the $\Mc$ class for skip-free chains as introduced in \cite{Choi-Patie}, such that the kernels in $\GMc$ has real and distinct eigenvalues and the time-reversal $\widehat{P}$ intertwines with a birth-death chain with the link kernel $\Lambda$ being related to the Siegmund kernel $H_S(x,y) = \1_{\{x \leq y\}}$.


\begin{definition}[The $\GMc$ class]\label{def:GMcclass}
	We say that, for some $\r\geq 3$,  $P \in \GMc_{\r}$ if $P \in \Sf$ with $\mathcal{X}=\llbracket 0,\r \rrbracket$ and for every $x \in \llbracket 0,\r-1 \rrbracket$, its time-reversal $(X,\widehat{\mP})$  satisfies
	\begin{enumerate}
		\item(stochastic monotonicity) $\widehat{\mP}_{x+1}(X_1 \leq x)\leq \widehat{\mP}_{x}(X_1 \leq x) $\,,
		\item(strict stochastic monotonicity) $\widehat{\mP}_{x+1}(X_1 \leq x-1) <\widehat{\mP}_{x}(X_1 \leq x-1), \quad x\neq 0, \quad \text{and}$
		\item(strict stochastic monotonicity) $\widehat{\mP}_{x+1}(X_1 \leq x+1) < \widehat{\mP}_{x}(X_1 \leq x+1), \quad x\neq \r-1, \quad \text{and}$
		\item(restricted downward jump) $\widehat{\mP}_{x+1}(X_1 \leq x-k) = \widehat{\mP}_x(X_1 \leq x-k), \quad k \in \llbracket 2,x \rrbracket, \quad \text{and}$
		\item(restricted upward jump) $\widehat{\mP}_{x+1}(X_1 \leq x+k) = \widehat{\mP}_x(X_1 \leq x+k), \quad k \in \llbracket 2,\r-1-x \rrbracket.$
	\end{enumerate}	
	Moreover, we say  $X \in \GMc^+_{\r}$ if $X \in \GMc_{\r}$ and for every $x \in \llbracket 0,\r-1 \rrbracket$,
	\begin{enumerate}[resume]
		\item(lazy Siegmund dual)
		$\widehat{\mP}_x(X_1 \leq x) - \widehat{\mP}_{x+1}(X_1 \leq x) \geq \dfrac{1}{2}.$
	\end{enumerate}
	When there is no ambiguity of the state space, we write $\GMc = \GMc_{\r}$ (resp.~$\GMc^+ = \GMc_{\r}^+$). Note that the upper-script of the plus sign in $\GMc^+$ means that this class has non-negative eigenvalues, see Remark \ref{rk:gmcplus} below.
\end{definition}

\begin{rk}\label{rk:Mc}
	Recall that in \cite{Choi-Patie}, if $P \in \Mc$, that is, $P$ is upward skip-free and satisfies $(1),(3),(5)$, then it is clear that $\Mc \subset \GMc$, as item $(2)$ and $(4)$ in Definition \ref{def:GMcclass} are automatically satisfied since the time-reversal $\widehat{P}$ is downward skip-free.
\end{rk}



Next, we formally state that a transformation of $P \in \GMc$ is contained in the similarity orbit of an irreducible birth-death kernel. The proof can be found in Section \ref{subsec:proofgmc}. For any square matrix $M$ on $\llbracket 0,\r \rrbracket$, we write $M^{\llbracket 0,\r-1 \rrbracket}$ to be the principal submatrix on $\llbracket 0,\r-1 \rrbracket$.
\begin{theorem}\label{thm:Mc}
	Let $P \in \GMc$ and write $M := P \Lambda$ on $\llbracket 0,\r \rrbracket$, where $\Lambda = (H_S^T D_{\pi})^{-1}$ and $D_{\pi}$ is the diagonal matrix of $\pi$.  Then $(\Lambda^{-1})^{\llbracket 0,\r-1 \rrbracket}(P \Lambda)^{\llbracket 0,\r-1 \rrbracket} \in \mathcal{S}(Q)$ with $Q$ being an irreducible birth-death transition kernel on  $\llbracket 0,\r -1\rrbracket$.
\end{theorem}

\begin{rk}[On the connection to the strong stationary duality theory by \cite{FD}]
	In this remark, we would like to highlight the connection between Theorem \ref{thm:Mc} and the classical construction of strong stationary duality (SSD) proposed by Diaconis and Fill. In \cite[Theorem $5.5$]{FD}, writing $\pi_0$ to be the initial distribution at time $0$, if $\pi_0(x)/\pi(x)$ is decreasing in $x$ and the time-reversal is stochastically monotone, then the SSD of a chain $X$ can be derived as the Doob $H$-transform of the Siegmund dual of the time-reversal of $X$, with $H = H_S^T \pi$ being the cumulative distribution function of $\pi$. 	As we shall see in the proof of Theorem \ref{thm:Mc}, our result bears a resemblance to the above construction by Diaconis and Fill, with $Q$ in Theorem \ref{thm:Mc} being a Doob transform of the state-restriction of the Siegmund dual of the time-reversal. Note that both our result and the classical SSD construction require stochastic monotonicity of the time-reversal, yet our $\GMc$ class requires more conditions (namely item (2) to (4) in Definition \ref{def:GMcclass}). While we apply the same Doob $H$-transform to the Siegmund dual, we apply a further Doob $\widetilde h$-transform for the state-restricted and Doob $H$-transformed Siegmund dual restricted to $\llbracket 0,\r - 1\rrbracket$, where $\widetilde h$ is defined in \eqref{eq:widetildeh} below.
\end{rk}

We now give an example that illustrates the $\GMc$ class.

	\begin{example}
		$$\widehat{P} = \begin{pmatrix}
		0.5 & 0.35 & 0.05 & 0.1 \\
		0.3 & 0.5 & 0.1 & 0.1 \\
		0.2 & 0.1 & 0.5 & 0.2 \\
		0.2 & 0.05 & 0.25 & 0.5 \\
		\end{pmatrix}, \quad P = \begin{pmatrix}
		0.5 & 0.2629 & 0.1157 & 0.1213 \\
		0.3994 & 0.5 & 0.0660 & 0.0346 \\
		0.0864 & 0.1515 & 0.5 & 0.2621 \\
		0.1648 & 0.1444 & 0.1907 & 0.5 \\
		\end{pmatrix}$$ has eigenvalues $1, 0.54, 0.28, 0.18$, and satisfies $(1)-(6)$ in Definition \ref{def:GMcclass}. Note that
		$$(\Lambda^{-1})^{\llbracket 0,\r-1 \rrbracket}(P \Lambda)^{\llbracket 0,\r-1 \rrbracket} = \begin{pmatrix}
		0.2 & 0.1 & 0 \\
		0.05 & 0.5 & 0.05 \\
		0 & 0.1 & 0.3
		\end{pmatrix}.$$
\end{example}

We proceed to investigate the separation cutoff phenomenon for the $\GMc$ class. For birth-death chains, they have been studied in \cite{DSC06} and \cite{CSC15} while it has recently been extended to upward skip-free chains by \cite{MZZ16} and \cite{Choi-Patie}. 
In order to establish the famous ``spectral gap times mixing time" criteria (see e.g. \cite{Peres04}) for this class, we will build upon the result of \cite{Fill} to first analyze the fastest strong stationary time of this class, followed by demonstrating that the proof in \cite{DSC06} carries over for this class of non-reversible chains.

To this end, we recall the definition of separation distance of Markov chains, which is used as a standard measure for convergence to equilibrium.
For $n \in \N$, the maximum separation distance $s(n)$ is defined by
$$s(n) = \max_{x,y \in E} \left[ 1 - \dfrac{P^n(x,y)}{\pi(y)}\right] = \max_{x \in E}\, \mathrm{sep}(P^n(x,\cdot),\pi) = \max_{x \in E} s_x(n).$$
One of its nice features is its connection to  strong stationary times that we now describe. We say that a randomized stopping time $T$ for a Markov chain $X$ with stationary distribution $\pi$ is a strong stationary time $T$, possibly depending on the initial starting position $x$, if, for all $x,y \in E$,
$$\mP_x(T = n, X_T = y) = \mP_x(T = n) \pi(y).$$
It is well-known, see e.g. \cite[Lemma $6.11$]{LPW09}, that the tail probability of a strong stationary time $U$ provides an upper bound on the separation distance, that is,
$$s_x(n) \leq \mP(U > n).$$
The fastest strong stationary time $T$ is a strong stationary time such that for all $n \in \N$, $s_x(n) = \mP(T > n)$.
We now provide a description of the cutoff phenomenon for Markov chains. Recall that the separation mixing times are defined, for any $x \in E$ and $\epsilon>0$, as \[T^s(x,\epsilon) = \min \{ n \geq 0;~\mathrm{sep}(P^n(x,\cdot),\pi) \leq \epsilon\}\]
and	\[T^s(\epsilon) = \min \{ n \geq 0;~s(n) \leq \epsilon\}.\]
A family, indexed by $n \in \N$, of ergodic chains $X^{(n)}$ defined on $\mathcal{X}_n = \llbracket 0,\r_n \rrbracket$ with transition matrix $P_n$, stationary distribution $\pi_n$ and separation mixing times ${\rm{T}}_n(\epsilon)=T^s_n(\epsilon)$ or $T^s_n(x,\epsilon)$, for some $x\in E$, is said to present a separation cutoff if there is a positive sequence  $(t_n)$ such that for all $\epsilon \in (0,1)$,
$$\lim_{n \to \infty} \dfrac{\rm{T}_n(\epsilon)}{t_n} = 1.$$
The family has a $(t_n, b_n)$ separation cutoff if the sequences $(t_n)$ and $(b_n)$ are positive, $b_n/t_n \to 0$ and
for all $\epsilon \in (0,1)$,
$$\limsup_{n \to \infty} \dfrac{|{\rm{T}}_n(\epsilon) - t_n|}{b_n} < \infty.$$

We now proceed to discuss the main results of this Section, with Theorem \ref{thm:sepcut} addressing the case of discrete time family of Markov chains and Theorem \ref{thm:sepcutcont} discussing the continuized version. Recall that the notation $\GMc^+$ introduced in Definition \ref{def:GMcclass} represents the generalized monotonicity class with \textit{non-negative} eigenvalues. This is an important subclass since the eigenvalues of the transition kernel (resp.~negative of the generator) are the parameters in the geometric distribution (resp.~exponential distribution) of the fastest strong stationary time in Theorem \ref{thm:sepcut} (resp.~Theorem \ref{thm:sepcutcont}).

\begin{theorem}\label{thm:sepcut}
	For $n \geq 1$, suppose that $P_n \in \GMc^+_{\r_n}$ on the state space $\mathcal{X}_n = \llbracket 0,\r_n \rrbracket$ that started at $0$. Let $(\theta_{n,i})_{i=1}^{\r_n}$ be the non-zero eigenvalues of $I - P_n$, and $(c_{n,i})_{i=0}^{\r_n}$ to be the mixture weights of the $n^{\mathrm{th}}$ chain defined in \eqref{eq:mixweights} in Lemma \ref{lem:fsstgen}. Define
	$$ w_{n,i} := \sum_{j \geq i}^{\r_n} c_{n,j} \,, \quad  t_n := \sum_{i=1}^{\r_n} \dfrac{w_{n,i}}{\theta_{n,i}}\,, \quad \underline{\theta}_n := \min_{1 \leq i \leq \r_n} \theta_{n,i}\,, \quad \rho_n^2 := \sum_{i=1}^{\r_n} w_{n,i}^2 \dfrac{1-\theta_{n,i}}{\theta_{n,i}^2}\,.$$
	Then this family has a separation cutoff if and only if $t_n \underline{\theta}_n \to \infty$. Furthermore, if $t_n \underline{\theta}_n \to \infty$, then there is a $(t_n, \max\{\rho_n,1\})$ separation cutoff.
\end{theorem}

\begin{rk}
	For discrete-time stochastically monotone birth-death chains which start at $0$, we have $w_i = 1$ for $i \in \llbracket 1,\r_n\rrbracket$ and $c_{n,0} = 0$, and hence we recover \cite[Theorem $5.2$]{DSC06}.
\end{rk}

\begin{theorem}\label{thm:sepcutcont}
	For $n \geq 1$, suppose that $L_n = P_n - I$ is the infinitesimal generator with $P \in \GMc^+_{\r_n}$ on the state space $\mathcal{X}_n = \llbracket 0,\r_n \rrbracket$ that started at $0$. Let $(\theta_{n,i})_{i=1}^{\r_n}$ be the non-zero eigenvalues of $-L_n$, and $(c_{n,i})_{i=0}^{\r_n}$ to be the mixture weights defined in \eqref{eq:mixweightscont} in Remark \ref{rk:fsstgencont}. Define
	$$ w_{n,i} := \sum_{j \geq i}^{\r_n} c_{n,j} \,, \quad  t_n := \sum_{i=1}^{\r_n} \dfrac{w_{n,i}}{\theta_{n,i}}\,, \quad \underline{\theta}_n := \min_{1 \leq i \leq \r_n} \theta_{n,i}\,, \quad \rho_n^2 := \sum_{i=1}^{\r_n} \dfrac{w_{n,i}^2}{\theta_{n,i}^2}\,.$$
	Then this family has a separation cutoff if and only if $t_n \underline{\theta}_n \to \infty$. Furthermore, if $t_n \underline{\theta}_n \to \infty$, then there is a $(t_n, \rho_n)$ separation cutoff.
\end{theorem}

We will only prove Theorem \ref{thm:sepcut} as the proof of Theorem \ref{thm:sepcutcont} is very similar and thus omitted.

\subsection{Proof of Theorem \ref{thm:Mc}}\label{subsec:proofgmc}
We write  $\widetilde{P}$ the so-called Siegmund dual (or $H_S$-dual) of $\widehat{P}$. That is, $\widetilde{P}^T = H_S^{-1} \widehat{P} H_S$ where  $H_S = (H_S(x,y))_{x,y \in \mathcal{X}}$ is defined to be
$H_S(x,y) = \1_{\{x \leq y\}}$
and its inverse $H_S^{-1}=(H^{-1}_S(x,y))_{x,y \in \mathcal{X}}$ is
$H_S^{-1}(x,y) = \1_{\{x = y\}} - \1_{\{x = y - 1\}}$, see \cite{Sieg76}. Since $X \in \GMc$, then $\widehat{P}$ is stochastically monotone, hence from \cite[Proposition 4.1]{ASM03}, we have that $\widetilde{P}$ is a sub-Markovian kernel. For $x \in \llbracket 0,\r-2 \rrbracket$, condition $2$ and $3$ in $\GMc$ yield, respectively, $\widetilde{p}(x,x+1) > 0$, while for $x \in \llbracket 1,\r-1 \rrbracket$, we have $\widetilde{p}(x,x-1) > 0$. Condition $4$ and $5$ in $\GMc$ guarantee that $\widetilde{p}(x,y) = 0$ for each $x \in \llbracket 0,\r-3 \rrbracket$ and $y \in \llbracket x+2,\r-1 \rrbracket$ and for each $x \in \llbracket 2,\r-1 \rrbracket$ and $y \in \llbracket 0,x-2 \rrbracket$.
That is, $\widetilde{P}$ is a (strictly substochastic) irreducible birth-death chain when restricted to the state space $\llbracket 0,\r-1 \rrbracket$. Denote $\widetilde{P}^{bd}$  the restriction of $\widetilde{P}$ to $\llbracket 0,\r-1 \rrbracket$. By breaking off the last row and last column of $\widetilde{P}$, we can write
\begin{align}\label{eq:hatPbd}
\widetilde{P} = \left( \begin{array}{cc}
\widetilde{P}^{bd} & \mathbf{v}  \\
\mathbf{0} & 1  \\
\end{array} \right) = (H_S^{-1} \widehat{P} H_S)^T ,
\end{align}
where $\mathbf{0}$ is a row vector of zero, and $\mathbf{v}$ is a column vector storing $\widetilde{p}(x,\r)$ for $x \in \llbracket 0,\r-1 \rrbracket$.
Considering the $h$-transform of $\widetilde{P}$ with $h = H_S^T \pi > \mathbf{0}$, see e.g.~\cite[Theorem $2$]{HM11}, we see that
$$M = P \Lambda = \Lambda \widetilde{P},$$
where $\Lambda = (H_S^T D_{\pi})^{-1}$ ($D_{\pi}$ is the diagonal matrix of $\pi$). Observing that the last row of $\widetilde{P}$ is zero except the last entry, we have
$$M^{\llbracket 0,\r-1 \rrbracket} = \Lambda^{\llbracket 0,\r-1 \rrbracket} \widetilde{P}^{bd}.$$
Note that $\widetilde{P}^{bd}$ is a strictly substochastic matrix with $\r$ as a killing boundary. Denote $\widetilde{T}^{bd}$ to be the lifetime of Markov chain with transition kernel $\widetilde{P}^{bd}$. However,  defining, with the obvious notation, for any $x \in \llbracket 0,\r-1 \rrbracket$,
\begin{align}\label{eq:widetildeh}
	\widetilde h(x) = \mathbb{P}_x(\widetilde{T}^{bd}_{\r -1} < \widetilde{T}^{bd}),
\end{align}
 we have, according to \cite[Theorem 3.1]{Choi-Patie}, that $\widetilde h$ is an harmonic function for $\widetilde{P}^{bd}$, i.e.~$\widetilde{P}^{bd} \widetilde h =\widetilde h$.  Hence, a standard result in Martin boundary theory, see e.g.~\cite[Theorem 2.2]{Choi-Patie}, entails that the Markov chain with  transition kernel $Q$, defined on $\llbracket 0,\r-1 \rrbracket\times \llbracket 0,\r-1 \rrbracket$ by  $Q(x,y) = \frac{\widetilde h(y)}{\widetilde h(x)}\widetilde{P}^{bd}(x,y)$, is an ergodic birth-death chain, which completes the proof.
\begin{rk}\label{rk:gmcplus}
	Note that condition $(6)$ in $\GMc^+$ guarantees that $\widetilde{P}$ is a lazy chain, that is $\widetilde{P}(x,x) \geq 1/2$ for all $x \in \mathcal{X}$, and hence the class $\GMc^+$ possesses non-negative eigenvalues.
\end{rk}

\subsection{Proof of Theorem \ref{thm:sepcut}}
	
Following the plan as outlined above in Section \ref{sec:sepcut}, we first analyze the distribution of the fastest strong stationary time of the class $\GMc^+$ in Lemma \ref{lem:fsstgen}, followed by detailing the proof of Theorem \ref{thm:sepcut}.

\begin{lemma}\label{lem:fsstgen}
	Suppose that $X$ is an ergodic Markov chain on the state space $\mathcal{X} = \llbracket 0,\r \rrbracket$ (and $\r \geq 3$) with transition matrix $P$ and stationary distribution $\pi$ which starts at $0$. If $P \in \GMc^+$, then the fastest strong stationary time is distributed as the $\mathbf{c}$-mixture of convolution of geometric $\sum_{k=1}^\r c_k \mathcal{G}(\lambda_1,\ldots, \lambda_k)$, where $i,j,k \in \llbracket 0,\r \rrbracket,$
	\begin{align}\label{eq:mixweights}
	Q_k := \dfrac{(P - \lambda_1 I)\ldots(P - \lambda_{k}I)}{(1-\lambda_1) \ldots (1-\lambda_{k})}\,, \quad \Gamma(i,j) := Q_i(0,j) \,, \quad c_k := \dfrac{\Gamma(k,\r) - \Gamma(k-1,\r)}{\pi(\r)}\,,
	\end{align}
	$\{\lambda_k\}_{k=1}^\r$ are the non-unit eigenvalues of $P$ in non-decreasing order and $\mathcal{G}(\lambda_1,\ldots, \lambda_k)$ is the convolution of geometric distributions with success probabilities $1-\lambda_1,\ldots,1-\lambda_k$ respectively.
\end{lemma}

\begin{rk}
	We alert the readers that $Q_k$ are the so-called spectral polynomials and $\Gamma(i,j)$ is $\Lambda(i,j)$ of \cite[Theorem $5.2$]{Fill} (since $\Lambda$ is used as the link kernel throughout this paper).
\end{rk}

\begin{rk}[On the fastest strong stationary time of $P$ and the absorption time of $Q$]
	In this remark, we would like to highlight the connection between the fastest strong stationary time $T$ of $P$ and the absorption time to $\r$ of $Q$.	According to Lemma \ref{lem:fsstgen}, if $P \in \GMc^+$, then $T$ is distributed as $\sum_{k=1}^\r c_k \mathcal{G}(\lambda_1,\ldots, \lambda_k)$. On the other hand, according to Theorem \ref{thm:Mc}, $Q$ is an irreducible birth-death process on $\llbracket 0,\tau-1 \rrbracket$, then using \cite[Theorem 1.1]{Miclo} there exists a probability measure such that the absorption time to $\tau$ of $Q$ starting from $0$ is distributed as, in our notations, $\sum_{k=1}^\r a_k \mathcal{G}(\lambda_1,\ldots, \lambda_k)$. In these two distributions, the same eigenvalues appear as parameters in the geometric distributions. See \cite{Miclo} for further connections with the class of phase-type distribution.
\end{rk}

\begin{proof}
	Suppose that $P \Lambda = \Lambda Q$.
	In view of \cite{Fill} Theorem $5.2$, it suffices to show that the $c_k \geq 0$. First, we show that $(Q - \lambda_1 I)\ldots(Q - \lambda_k I)$ are non-negative matrices, where $Q$ is the Siegmund dual of $\widehat{P}$. We will prove this via induction on $k$. For $k = 1$, thanks to \cite[Theorem $3.2$]{MW79}, we have $Q^{BD} - \lambda_1 I \geq \mathbf{0}$, where $Q^{BD} := Q^{\llbracket 0,\r-1 \rrbracket}$ is the restriction of $Q$ except the last row and column, which leads to
	\begin{align*}
	Q - \lambda_1 I = \left( \begin{array}{cc}
	Q^{BD} - \lambda_1 I & \mathbf{h}  \\
	\mathbf{0}^T & 1 - \lambda_1  \\
	\end{array} \right) \geq \mathbf{0}\,.
	\end{align*}
	Suppose that
	\begin{align*}
	\prod_{i=1}^k (Q - \lambda_i I) = \left( \begin{array}{cc}
	\prod_{i=1}^k (Q^{BD} - \lambda_i I) & \mathbf{n}  \\
	\mathbf{0}^T & \prod_{i=1}^k (1 - \lambda_i)  \\
	\end{array} \right) \geq \mathbf{0}\,,
	\end{align*}
	where $\mathbf{n} \geq \mathbf{0}$ is a non-negative vector. Therefore,
	\begin{align*}
	\prod_{i=1}^{k+1} (Q - \lambda_i I) = \left( \begin{array}{cc}
	\prod_{i=1}^{k+1} (Q^{BD} - \lambda_i I) & \prod_{i=1}^{k} (Q^{BD} - \lambda_i I)\mathbf{h} + (1-\lambda_{k+1})\mathbf{n}  \\
	\mathbf{0}^T & \prod_{i=1}^{k+1} (1 - \lambda_i)  \\
	\end{array} \right) \geq \mathbf{0}\,,
	\end{align*}
	which completes the induction by using \cite[Theorem $3.2$]{MW79} again on $\prod_{i=1}^{k+1} (Q^{BD} - \lambda_i I)$. Define $$Z_k :=  (H_S^T)^{-1} \prod_{i=1}^k \dfrac{Q - \lambda_i I}{1-\lambda_i} H_S^T\,.$$ Note that
	$P = D_{\pi}^{-1} (H_S^T)^{-1} Q H_S^T D_{\pi}$, so
	$c_k \geq 0$ if and only if
	$ Z_k(0,\r) - Z_{k-1}(0,\r) \geq 0$ if and only if (here we make use of $H_S^T$)
	$$  \left(\prod_{i=1}^k \dfrac{Q - \lambda_i I}{1-\lambda_i}\right)(0,\r) - \left(\prod_{i=1}^{k-1} \dfrac{Q - \lambda_i I}{1-\lambda_i} \right) (0,\r) = \left(\prod_{i=1}^{k-1} \dfrac{Q^{BD} - \lambda_i I}{1-\lambda_i}\mathbf{h}\right)(0) \geq 0\,,$$
	which is true.
\end{proof}
When we have a handle on the fastest strong stationary time, we can then analyze the separation cutoff phenomenon, and the rest of the proof follow the Chebyshev inequality framework introduced by \cite{DSC06}. More precisely, denote $P_n^k$ to be the distribution of the $n^{\mathrm{th}}$ chain at time $k$, $\pi_n$ to be the stationary measure and $T_n$ to be the fastest strong stationary time of the $n^{\mathrm{th}}$ chain. We note that $\E(T_n) = t_n$ and $\mathrm{Var}(T_n) = \rho_n^2$. The key to the proof is the following:
	\begin{align*}
	\rho_n^2 = \underline{\theta}_n^{-2} \sum_{i=1}^{\r_n} w_{n,i}^2 \dfrac{\left(1-\theta_{n,i}\right) \underline{\theta}_n^2}{\theta_{n,i}^2} \leq \underline{\theta}_n^{-2}\sum_{i=1}^{\r_n} w_{n,i} \dfrac{\underline{\theta}_n}{\theta_{n,i}} = \underline{\theta}_n^{-1} t_n\,,
	\end{align*}
	where we use $\theta_{n,i} \geq 0$, $\underline{\theta}_n/\theta_{n,i} \leq 1$ and $w_i \leq 1$ in the first inequality. The rest of the proof follows as that of \cite[Theorem $8.1$]{Choi-Patie}, which does not require reversibility of the chain.	
	
\begin{rk}\label{rk:fsstgencont}
	The corresponding result of Lemma \ref{lem:fsstgen} in the continuous-time setting is stated in the following in order to prove Theorem \ref{thm:sepcutcont}. Suppose that $X$ is a continuous-time ergodic Markov chain on the state space $\mathcal{X} = \llbracket 0,\r \rrbracket$ (and $\r \geq 3$) with generator $L = P - I$ and stationary distribution $\pi$ which starts at $0$. If $P \in \GMc^+$, then the fastest strong stationary time is distributed as the $\mathbf{c}$-mixture of convolution of exponential $\sum_{k=1}^\r c_k \mathcal{E}(\theta_1,\ldots, \theta_k)$, where $i,j,k \in \llbracket 0,\r \rrbracket,$
		\begin{align}\label{eq:mixweightscont}
		Q_k := \dfrac{(L + \theta_1 I)\ldots(L + \theta_{k}I)}{\theta_1 \ldots \theta_{k}}\,, \quad \Gamma(i,j) := Q_i(0,j) \,, \quad c_k := \dfrac{\Gamma(k,\r) - \Gamma(k-1,\r)}{\pi(\r)}\,,
		\end{align}
		and $\{\theta_k\}_{k=1}^\r$ are the non-zero eigenvalues of $-L$ in non-increasing order and $\mathcal{E}(\theta_1,\ldots, \theta_k)$ is the convolution of exponential distributions with mean $1/\theta_1,\ldots,1/\theta_k$ respectively.
	
\end{rk}

\section{${\rm{L}}^2$-cutoff}\label{sec:l2cut}

The aim of this Section is to investigate the spectral criterion for the existence of ${\rm{L}}^2$-cutoff for the class of Markov chains in a continuous-time setting with generator $L$ and similarity on the generator level. That is, in the notation of Definition \ref{def:sim} and \ref{def:Sclass}, $L \in \mathcal{S}(G)$, where $G$ is a reversible generator. We denote the spectral gap $\lambda = \lambda(L)$ of $L$ by
\begin{align}\label{eq:specgap}
\lambda = \lambda(L) = \inf\{\langle -Lf,f \rangle_{\pi};\, f \in \mathrm{Dom}(L), \mathrm{real}\,\mathrm{valued}, \E_{\pi}(f) = 0, \E_{\pi}(f^2) = 1\}.
\end{align}
This follows and generalizes the work of \cite{CSC07,CSC10,CHS17} who studied the ${\rm{L}}^2$-cutoff phenomena in the context of normal Markov processes.
Adapting the notations therein, we proceed to provide a quick review on the notion of ${\rm{L}}^2$-cutoff. 

\begin{definition}\label{def:cutoff}
	For $n \geq 1$, let $g_n: [0,\infty) \mapsto [0,\infty]$ be a non-increasing function vanishing at infinity. Assume that
	$$M = \limsup_{n \to \infty} g_n(0) > 0.$$
	Then the family $\mathcal{G} = \{g_n: n \geq 1\}$ is said to have
	\begin{enumerate}
		\item A \textit{cutoff} if there exists a sequence of positive numbers $t_n$, known as the cutoff time, such that for $\epsilon \in (0,1)$,
		$$\lim_{n \to \infty} g_n((1+\epsilon)t_n) = 0, \quad \lim_{n \to \infty} g_n((1-\epsilon)t_n) = M.$$
		\item A $(t_n,b_n)$-\textit{cutoff} if $t_n > 0$, $b_n > 0$, where $b_n$ is known as the cutoff window, $b_n = o(t_n)$ and
		$$\lim_{c \to \infty} \limsup_{n \to \infty} g_n(t_n + c b_n) = 0, \quad \lim_{c \to -\infty} \liminf_{n \to \infty} g_n(t_n - c b_n) = M.$$
	\end{enumerate}
\end{definition}
If $\eta P_t \ll \pi$ with density $f(t,\eta,\cdot)$, then the chi-squared distance is given by
$$D_2(\eta,t)^2 = \int_{\mathcal{X}} |f(t,\eta,x) - 1|^2 \, \pi(dx).$$
Suppose that we have a family of measurable spaces $(\mathcal{X}_n,\mathcal{B}_n)_{n \in \mathbb{N}}$. For $n \in \mathbb{N}$, we denote $p_n(t,\eta_n,\cdot)$ defined on $(\mathcal{X}_n,\mathcal{B}_n)$ to be the transition function with initial probability law $\eta_n \ll \pi_n$ and $t \geq 0$. We denote $f_n$ to be the ${\rm{L}}^2$-density of $\eta_n$ with respect to $\pi_n$. The family $\{p_n(t,\eta_n,\cdot): t \in [0,\infty) \}$ has an ${\rm{L}}^2$-cutoff (resp.~$(t_n,b_n)$ ${\rm{L}}^2$-cutoff) if $\{g_n(t) = D_{n,2}(\eta_n,t): n \geq 1\}$ has a cutoff (resp.~$(t_n,b_n)$-cutoff) as in Definition \ref{def:cutoff}, where $D_{n,2}(\eta_n,t)$ is the chi-squared distance of the $n^{th}$ chain.

Our main result in Theorem \ref{thm:cutoff} gives the spectral criterion for ${\rm{L}}^2$-cutoff to the family of process with $L_n \in \S(G_n)$, where $G_n$ is a reversible generator. We denote the (non-self-adjoint) spectral measure of $L_n$ of the $n^{th}$ chain by $F_{n,B}$ for $B \in \mathcal{B}(\mathbb{C})$, and $H_{n,B} = F_{n,B} F_{n,B}^*$. We use the following notation: for $\delta, C > 0$ and $B \in \mathcal{B}(\mathbb{C})$, we set, for any $n \in \mathbb{N}$,
\begin{eqnarray*}
	V_n(B) &=& \langle H_{n,B} f_n, f_n \rangle_{\pi_n}, \quad t_n(\delta) = \inf \{t : D_{n,2}(\eta_n,t) \leq \delta\}, \\
	\lambda_n(C) &=& \inf \{\lambda:  V_n([\lambda_n,\lambda]) > C\}, \quad 	\tau_n(C) = \sup\bigg\{\dfrac{\log(1 +  V_n([\lambda_n,\lambda]) )}{2\lambda} : \lambda \geq \lambda_n(C)\bigg\}, \\
	\gamma_n &=& \lambda_n(C)^{-1} \textrm{ and }
	b_n = \lambda_n(C)^{-1} \log(\lambda_n(C)\tau_n(C)).
\end{eqnarray*}

\begin{theorem}\label{thm:cutoff}
	Suppose that $L_n \in \S(G_n)$ for each member in the family $\{p_n(t,\eta_n,\cdot): t \in [0,\infty) \}$, where $G_n$ is a reversible generator. If $\pi_n(f_n^2) \to \infty$, then the following are equivalent.
	\begin{enumerate}
		\item $\{p_n(t,\eta_n,\cdot): t \in [0,\infty) \}$ has an ${\rm{L}}^2$-cutoff.
		\item For some positive constants $C, \delta, \epsilon$,
		$$\lim_{n \to \infty} t_n(\delta) \lambda_n(C) = \infty, \quad \lim_{n \to \infty} \int_{[\lambda_n, \lambda_n(C)]} e^{-\epsilon \gamma t_n(\delta)} dV_n(\gamma) = 0.$$
		\item For some positive constants $C, \epsilon$,
		$$\lim_{n \to \infty} \tau_n(C) \lambda_n(C) = \infty, \quad \lim_{n \to \infty} \int_{[\lambda_n, \lambda_n(C)]} e^{-\epsilon \gamma \tau_n(C)} dV_n(\gamma) = 0.$$
	\end{enumerate}
	If $(2)$ (resp.~$(3)$) holds, then $\{p_n(t,\eta_n,\cdot): t \in [0,\infty) \}$ has a $(t_n(\delta),\gamma_n)$ ${\rm{L}}^2$-cutoff (resp.~$(\tau_n(C),b_n)$ ${\rm{L}}^2$-cutoff).
\end{theorem}

\begin{rk}\label{lem:cutoffrk}
	If $L_n$ is reversible, then $H_{n,B} = F_{n,B} F_{n,B}^* = F_{n,B}^2 = F_{n,B}$ since $F_{n,B}$ is a self-adjoint projection in this case. The above result then retrieves exactly Theorem $4.6$ of \cite{CSC10}.
\end{rk}

\subsection{Proof of Theorem \ref{thm:cutoff}}
To prove Theorem \ref{thm:cutoff}, it relies on the following lemma that relates the chi-squared distance to the spectral decomposition of the infinitesimal generator $-L$, which allows us to connect with the Laplace transform of the spectral measure $H_B = F_B F_B^*$.

\begin{lemma}\label{lem:L2spectraldecomp}
	Let $X$ be a Markov process with $X_0 \sim \eta$, generator $L \in \S(G)$, where $G$ is a reversible generator, such that $\eta \ll \pi$ with ${\rm{L}}^2(\pi)$-density $f$ and spectral gap $\lambda > 0$. Denote $\{F_B: B \in \mathcal{B}(\mathbb{C})\}$ to be the non-self-adjoint spectral measure for $-L$, 
	and we define, for $B \in \mathcal{B}(\mathbb{C})$,
	$$H_B = F_B F_B^*.$$
	Then, for $t \geq 0$,
	\begin{align*}
	D_2(\eta,t)^2 = \int_{[\lambda,\infty)} e^{-2\gamma t}\, d \langle H_{\gamma} f, f \rangle_{\pi}\,.
	\end{align*}
\end{lemma}

\begin{proof}
	By the definition of chi-square distance $D_2$ and $\pi(f) = 1$, we have
	$$D_2(\eta,t)^2 = \norm{\widehat{P}_t f - \pi(f)}_{\pi}^2 
	= \int_{[\lambda,\infty)} e^{-2\gamma t}\, d \langle H_{\gamma} f, f \rangle_{\pi},$$
	where the last equality follows from \cite[second half of the proof of Lemma $3.19$ on page 1542]{IT14}.
\end{proof}
Lemma \ref{lem:L2spectraldecomp} reveals that the problem of ${\rm{L}}^2$-cutoff reduces to the cutoff phenomenon of the Laplace transform. We proceed to complete the proof of Theorem \ref{thm:cutoff}. By Lemma \ref{lem:L2spectraldecomp}, we take $g_n(t) = D_{n,2}(\eta_n,t)$
in Definition \ref{def:cutoff}, and the desired result follows from the Laplace transform cutoff criteria in Theorem $3.5$ and Theorem $3.8$ of \cite{CSC10}. Precisely, the chi-squared distance is of the form
$$D_{n,2}(\eta_n,t)^2 = \int_{[\lambda_n,\infty)} e^{-2\gamma t}\, d V_n(\lambda).$$
This is exactly the form of function considered in \cite[equation $(3.1)$]{CSC10}, and consequently we can invoke Theorem $3.5$ and Theorem $3.8$ of the aforementioned paper.

\begin{rk}
	As mentioned in Remark \ref{lem:cutoffrk}, if $L_n$ is reversible, then our Theorem \ref{thm:cutoff} retrieves exactly Theorem 4.6 of \cite{CSC10} whose  proof  is a combination of the results of theorems 3.5, 3.8 and 4.4 therein. The strategy of the proof is as follows: Theorem 4.4 claims that the chi-squared distance to stationarity of a reversible Markov process is a Laplace transform, thus the Laplace transform cutoff results of theorems 3.5 and 3.8 can be applied.
	In our proof of Theorem \ref{thm:cutoff}, we follow the same strategy. We first show Lemma 4.1 (which is the parallel version of  \cite[Theorem 4.4]{CSC10}), which states that for $L \in \mathcal{S}(G)$ with $G$ being reversible, the chi-squared distance to stationarity of $L$ is also a Laplace transform. Consequently, the Laplace transform cutoff results of theorems 3.5 and 3.8 of  \cite{CSC10} can be applied in our setting.
\end{rk}

\subsection{${\rm{L}}^p$-cutoff}

We proceed by investigating the ${\rm{L}}^p$-cutoff for fixed $p \in (1,\infty)$ for the class $\S$. Recall that 
Chen and Saloff-Coste \cite[Theorem $4.2,4.3$]{CSC07} have shown that for a family of \textit{normal} ergodic transition kernel $P_n$, the max-${\rm{L}}^p$ cutoff is equivalent to the {\emph{spectral gap times mixing time}} going to infinity. We can extend their result to the case of the non-normal chains in $ \S $ as follows, using similar techniques as in \cite{Choi-Patie} for the class of skip-free chains similar to birth-death chains.
\begin{theorem}[Max-${\rm{L}}^p$ cutoff]
	Suppose that, for each  $n \geq 1$, $L_n \in  \S(G_n)$ with $G_n$ being a reversible generator, transition kernel $P^t_n = e^{tL_n} \stackrel{\Lambda_n}{\sim} Q^t_n = e^{tG_n}$ on $\mathcal{X}_n$ and spectral gap of $G_n$ given by $\lambda_n = \lambda_n(G_n)$, where we recall the definition of spectral gap in \eqref{eq:specgap}. Assume that the condition numbers $\kappa(\Lambda_n)$ of the link kernels are uniformly bounded, that is,
	$$\sup_{n \geq 1} \kappa(\Lambda_n) < \infty.$$
	Fix $p \in (1,\infty)$ and $\epsilon > 0$. Consider the max-${\rm{L}}^p$ distance to stationarity $$f_n(t) = \sup_{x \in \mathcal{X}_n} \left(\int_{\mathcal{X}_n} |p_n(t,\delta_x,y) - 1|^p \, \pi_n(dy)\right)^{1/p} $$
	and define
	$$t_n = \inf\{t > 0;\: f_n(t) \leq \epsilon\}, \quad \mathcal{F} = \{f_n;\: n=1,2,\ldots\}.$$
	Assume that each $n$, $f_n(t) \to 0$ as $t \to \infty$ and $t_n \to \infty$. Then the family $\mathcal{F}$ has a max-${\rm{L}}^p$ cutoff if and only if $t_n \lambda_n \to \infty$. In this case there is a $(t_n, \lambda_n^{-1})$ cutoff.
\end{theorem}
The proof in \cite[Theorem $4.2,4.3$]{CSC07} works nicely as long as we have Lemma \ref{lem:l2lp} below, which gives a two-sided control on the ${\rm{L}}^p(\pi)$ norm of $P^t - \pi$. 
The following lemma is then the key to the proof.
\begin{lemma}\label{lem:l2lp}
	Suppose that $L \in  \S(G)$ with $G$ being a reversible generator, transition kernel $P^t = e^{tL} \stackrel{\Lambda}{\sim} Q^t = e^{tG}$ and the spectral gap of $G$ is $\lambda = \lambda(G)$, where we recall the definition of spectral gap in \eqref{eq:specgap}. Fix $p \in (1,\infty)$. Then, for any $t \geq 0$, we have
	\begin{align}
	2^{-1+\theta_p} e^{-\lambda t\theta_p} &\leq \norm{P^t - \pi}_{{\rm{L}}^p(\pi) \to {\rm{L}}^p(\pi)} \leq 2^{|1-2/p|} (\kappa(\Lambda) e^{-\lambda t})^{1-|1-2/p|}, \label{eq:lp}
	\end{align}
	where $\theta_p \in [1/2,1]$ and $\kappa(\Lambda)= \norm{\Lambda}_{{\rm{L}}^2(\pi_Q) \to {\rm{L}}^2(\pi)} \: \norm{\Lambda^{-1}}_{{\rm{L}}^2(\pi) \to {\rm{L}}^2(\pi_Q)}$.
\end{lemma}
\begin{proof}
	By the  Riesz-Thorin interpolation theorem, see e.g.~\cite[equation $3.4$]{CSC07}, we have
	$$\norm{P^t - \pi}_{{\rm{L}}^p(\pi) \to {\rm{L}}^p(\pi)} \leq 2^{|1-2/p|} \norm{P^t - \pi}_{{\rm{L}}^2(\pi) \to {\rm{L}}^2(\pi)}^{1-|1-2/p|},$$
	which when combined with Corollary \ref{cor:spectralexp} gives the upper bound of \eqref{eq:lp}. Next, to show the lower bound in \eqref{eq:lp}, we use another version of the Riesz-Thorin interpolation theorem,  see e.g.~\cite[Lemma $4.1$]{CSC07}, to  get
	$$\norm{P^t - \pi}_{{\rm{L}}^p(\pi) \to {\rm{L}}^p(\pi)} \geq2^{-1+\theta_p} \norm{P^t - \pi}_{{\rm{L}}^2(\pi) \to {\rm{L}}^2(\pi)}^{\theta_p} \geq 2^{-1+\theta_p} e^{-\lambda t\theta_p},$$
	where we use Corollary \ref{cor:spectralexp} in the second inequality. This completes the proof.
\end{proof}


\section{Non-asymptotic estimation error bounds for integral functionals}\label{sec:nonasympest}

In this Section, we would like to estimate integral functionals of the type
$$\Gamma_T(f) = \int_0^T f(X_t)\, dt, \quad T \geq 0,$$
where $T$ is a fixed time and $f$ is a function such that the integral $\Gamma_T(f)$ is well-defined. This follows the line of work of \cite{AC16}, who studied the same problem with the assumption that the infinitesimal generator of the Markov process is a normal operator. This type of integral functionals appear in a number of applications. For instance, if we take $f = \1_{B}$, the indicator function of the Borel set $B$, then $\Gamma_T(f)$ is the occupation time of the process in $B$. As another example, it is not hard to see that such functional appears in the study of path-dependent derivatives in mathematical finance, see e.g. \cite{MMM97}. In practice however, one often only have access to a sample path of the Markov process at discrete time point. A natural estimator for $\Gamma_T(f)$, known as the Riemann-sum estimator, is given by
$$\hat{\Gamma}_{T,n}(f) = \sum_{k=1}^n f(X_{(k-1)\Delta_n}) \Delta_n,$$
where we observe $(X_t)_{t \in [0,T]}$ at discrete epochs $t = (k-1)\Delta_n$ with $k \in \llbracket n \rrbracket$ and $\Delta_n = T/n$, with the idea that we approximate $\Gamma_T(f)$ by its Riemann-sum.

For a stationary Markov process and $f \in {\rm{L}}^2(\pi)$, both $\Gamma_T(f)$ and $\hat{\Gamma}_{T,n}(f)$ are $\pi$-a.s. defined everywhere in ${\rm{L}}^2(\mathbb{P})$. If $L \in \S(G)$, we identify by Riesz theorem a linear self-adjoint operator $A$ such that for $f,g \in {\rm{L}}^2(\pi)$,
$$\langle Af,g \rangle_{\pi} = \int_{\sigma(L)} |\lambda|^2 \, d\langle H_{\lambda}^* f, g \rangle_{\pi},$$
where we recall $H_{\lambda}^* = F_{\lambda}^* F_{\lambda}$ is a self-adjoint operator and $F_{\lambda}$ is the spectral measure of $-L$. For $s \geq 0$, we define the space $\mathcal{D}^s(A) = \mathrm{Dom}(A^s) \subset {\rm{L}}^2(\pi)$ by functional calculus on $A$ with the seminorm $\norm{f}_{\mathcal{D}^s(A)} = \norm{A^{s/2} f}_{\pi}$.

The main results are the following error bounds, in which the proof is similar as that of \cite[Theorem $2.2$, Corollary $2.3$, Theorem $2.4$]{AC16} and is deferred to Section \ref{subsec:proofofest}. Note that \eqref{eq:esterror2} gives the error bound on the space average of $X$ with the finite-time and finite-sample estimator $T^{-1}\hat{\Gamma}_{T,n}(f)$, while \eqref{eq:esterror3} offers the error bound for the non-stationary Markov process such that $X_0 \sim \eta$.

\begin{theorem}\label{thm:esterror}
	Let $X$ be a Markov process with $X_0 \sim \pi$ and generator $L \in \S(G)$. There exists a constant $C$ such that for all $T \geq 0$, $0 \leq s \leq 1$, $f \in \mathcal{D}^s(A)$, $f_0 \in \mathrm{Dom}(A^{-1})$ with $f_0 = f - \int f\, d\pi$,
	\begin{align}
	\norm{\Gamma_T(f) - \hat{\Gamma}_{T,n}(f)}_{{\rm{L}}^2(\mathbb{P})} &\leq C \sqrt{\norm{f}_{\mathcal{D}^s(A)} \norm{f}_{\pi} T \Delta_n^{1+s}}, \label{eq:esterror}\\
	\norm{T^{-1}\hat{\Gamma}_{T,n}(f) - \int f \, d\pi}_{{\rm{L}}^2(\mathbb{P})} &\leq \dfrac{C}{\sqrt{T}} \left(\sqrt{\norm{f}_{\mathcal{D}^s(A)} \norm{f}_{\pi} \Delta_n} + \sqrt{\norm{A^{-1}f_0}_{\pi} \norm{f_0}_{\pi}}\right). \label{eq:esterror2}
	\end{align}
	If $X_0 \sim \eta$ such that $\eta \ll \pi$ with density $d \eta/d\pi$, then there exists a constant $C$ such that for all $T \geq 0$, $0 \leq s \leq 1$ and $f \in \mathcal{D}^s(A)$,
	\begin{align}\label{eq:esterror3}
	\norm{\Gamma_T(f) - \hat{\Gamma}_{T,n}(f)}_{{\rm{L}}^2(\mathbb{P})} \leq C \norm{\dfrac{d \eta}{d \pi}}_{\infty,\pi}^{1/2} \sqrt{\norm{f}_{\mathcal{D}^s(A)} \norm{f}_{\pi} T \Delta_n^{1+s}},
	\end{align}
	where $\norm{\cdot}_{\infty,\pi}$ is the sup-norm in $L^{\infty}(\pi)$.
\end{theorem}

\begin{rk}
	When $L$ is reversible, then $A$ can be identified as $|L|^2$, where we can then retrieve the results of \cite{AC16}.
\end{rk}

\subsection{Proof of Theorem \ref{thm:esterror}}\label{subsec:proofofest}
	We first state a lemma (see \cite[first half of the proof of Lemma $3.19$ on page 1542]{IT14}) which will be used repeatedly in the proof.
	\begin{lemma}\label{lem:dsa}
		For $f \in \mathcal{D}^s(A)$,
		\begin{align}\label{eq:dsa}
		\left|\int_{\sigma(L)} \lambda^s \,d\langle F_{\lambda} f,f \rangle_{\pi}\right| \leq \left(\int_{\sigma(L)} |\lambda|^{2s} \, d\langle H_{\lambda}^* f, f \rangle_{\pi}\right)^{1/2} \norm{f}_{\pi} = \norm{f}_{\mathcal{D}^s(A)} \norm{f}_{\pi}.
		\end{align}
	\end{lemma}
	\begin{proof}
		For sake of completeness, we repeat the arguments of \cite[first half of the proof of Lemma $3.19$ on page 1542]{IT14}. Let $[\alpha,\beta]$ be a bounded interval and $(\Delta_k)_{k=1}^n$ be a family of disjoint intervals whose union is $[\alpha,\beta]$. For every $k$, we choose $\lambda_k \in \Delta_k$. Fix $f \in {\rm{L}}^2(\pi)$. For the Cauchy sums defining the integrals using triangle inequality and Cauchy-Schwartz inequality we have
		\begin{align*}
		\left|\sum_{k=1}^{n} \lambda_{k}^s \left\langle F_{\Delta_k} f,f\right\rangle_{\pi}\right| & \leq \sum_{k=1}^{n}\left|\lambda_{k}\right|^s |\left\langle F_{\Delta_k} f , F_{\Delta_k}^* f \right\rangle_{\pi} | \\ & \leq\left(\sum_{k=1}^{n} |\lambda_{k}|^{2s} \langle F_{\Delta_k} f, F_{\Delta_k} f \rangle_{\pi} \right)^{1 / 2}\left(\sum_{k=1}^{n} \langle F_{\Delta_k}^* f, F_{\Delta_k}^* f \rangle_{\pi}\right)^{1 / 2} \\ & \leq \left(\sum_{k=1}^{n} |\lambda_{k}|^{2s} \langle H_{\Delta_k}^* f,  f \rangle_{\pi} \right)^{1/2} \|f\|_{\pi}.
		\end{align*}
		The inequality \eqref{eq:dsa} holds on every finite interval and the desired result follows by taking limits.
	\end{proof}

	We now proceed to give the proof of Theorem \ref{thm:esterror}. We first prove \eqref{eq:esterror} and consider
	\begin{align*}
		\norm{\Gamma_T(f) - \hat{\Gamma}_{T,n}(f)}_{{\rm{L}}^2(\mathbb{P})}^2 &= \E \left[ \left(\sum_{k=1}^n \int_{(k-1)\Delta_n}^{k\Delta_n} \left(f(X_r) - f(X_{(k-1)\Delta_n})\right)\,dr \right)^2 \right] \\
		&= \sum_{k,l=1}^n \int_{(k-1)\Delta_n}^{k\Delta_n} \int_{(l-1)\Delta_n}^{l\Delta_n} \E \left[ \left(f(X_r) - f(X_{(k-1)\Delta_n})\right) \left(f(X_h) - f(X_{(l-1)\Delta_n})\right) \right]\,drdh\,,
	\end{align*}
	then we proceed to bound the diagonal ($k = l$) and off-diagonal ($k \neq l$) terms. For the diagonal terms, by stationarity we have for $(k-1)\Delta_n \leq r \leq h \leq k\Delta_n$,
	\[\E \left[ \left(f(X_r) - f(X_{(k-1)\Delta_n})\right) \left(f(X_h) - f(X_{(k-1)\Delta_n})\right) \right] = \langle(P_{h-r} - I)f + (I - P_{h - (k-1)\Delta_n})f + (I - P_{r - (k-1)\Delta_n})f,f\rangle_{\pi}, \]
	so by symmetry in $r$ and $h$ we have
	\begin{align*}
		&\quad \sum_{k=1}^n \int_{(k-1)\Delta_n}^{k\Delta_n} \int_{(k-1)\Delta_n}^{k\Delta_n} \E \left[ \left(f(X_r) - f(X_{(k-1)\Delta_n})\right) \left(f(X_h) - f(X_{(k-1)\Delta_n})\right) \right]drdh \\
		&= 2 n \bigg\langle \left( \int_0^{\Delta_n} \int_0^h (P_{h-r} - I)\,drdh + \Delta_n \int_0^{\Delta_n} (I - P_h) \, dh \right)f,f\bigg\rangle_{\pi} \\
		&= \langle \Phi(L)f,f \rangle_{\pi} \\
		&= \int_{\sigma(L)} \Phi(\lambda) \,d\langle F_{\lambda} f,f \rangle_{\pi},
	\end{align*}
	where the last equality follows from the functional calculus of $L$ in Theorem \ref{thm:mcins} and for $\lambda \in \sigma(L)$,
	\[ \Phi(\lambda) = 2n \left(\int_0^{\Delta_n} \int_0^h (e^{\lambda(h-r)} - 1)\,drdh + \Delta_n \int_0^{\Delta_n} (1 - e^{\lambda h}) \, dh\right).\]
	From \cite[Page $15$]{AC16}, we know that $|\Phi(\lambda)| \leq 4 n \Delta_n^{2+s} |\lambda|^s$ with fixed $0 \leq s \leq 1$. Now, we apply Lemma \ref{lem:dsa} to arrive at
	\[\left|\int_{\sigma(L)} \Phi(\lambda) \,d\langle F_{\lambda} f,f \rangle_{\pi}\right| \leq 4T\Delta_n^{1+s} \norm{f}_{\pi} \left(\int_{\sigma(L)} |\lambda|^{2s} \, d\langle H_{\lambda}^* f, f \rangle_{\pi}\right)^{1/2} = 4T\Delta_n^{1+s} \norm{f}_{\pi} \norm{f}_{\mathcal{D}^s(A)}.  \]
	Next, we bound the off-diagonal terms, in which
	 \begin{align*}
	 	&\quad 2\sum_{k > l}^n \int_{(k-1)\Delta_n}^{k\Delta_n} \int_{(l-1)\Delta_n}^{l\Delta_n} \E \left[ \left(f(X_r) - f(X_{(k-1)\Delta_n})\right) \left(f(X_h) - f(X_{(l-1)\Delta_n})\right) \right]drdh \\
	 	&= 2 \bigg\langle \left( \int_0^{\Delta_n} \int_0^{\Delta_n} \left(\sum_{k > l = 1}^n P_{(k-l)\Delta_n - r}\right)(P_{h} - I)(I - P_r)\,drdh \right)f,f\bigg\rangle_{\pi} \\
	 	&= \langle \tilde{\Phi}(L)f,f \rangle_{\pi} \\
	 	&= \int_{\sigma(L)} \tilde{\Phi}(\lambda) \,d\langle F_{\lambda} f,f \rangle_{\pi},
	 \end{align*}
	where the last equality follows again from the functional calculus of $L$ in Theorem \ref{thm:mcins} and for $\lambda \in \sigma(L)$,
	\[ \tilde{\Phi}(\lambda) = 2 \left(\int_0^{\Delta_n} \int_0^{\Delta_n} \left(\sum_{k > l =1}^n e^{\lambda((k-l)\Delta_n - r)}\right)(e^{\lambda h} - 1)(1 - e^{\lambda r})\,drdh \right).\]
	Using \cite[$(16)$]{AC16} there exists a universal constant $\tilde{C} < \infty$ such that
	$|\tilde{\Phi}(\lambda)| \leq \tilde{C} T \Delta_n^{1+s} |\lambda|^s$, and together with Lemma \ref{lem:dsa} yield
	\[\left|\int_{\sigma(L)} \tilde{\Phi}(\lambda) \,d\langle F_{\lambda} f,f \rangle_{\pi}\right| \leq \tilde{C} T\Delta_n^{1+s} \norm{f}_{\pi} \left(\int_{\sigma(L)} |\lambda|^{2s} \, d\langle H_{\lambda}^* f, f \rangle_{\pi}\right)^{1/2} = \tilde{C} T\Delta_n^{1+s} \norm{f}_{\pi} \norm{f}_{\mathcal{D}^s(A)}.  \]
	Next, we prove \eqref{eq:esterror2}. By \eqref{eq:esterror} and triangle inequality,
	\begin{align*}
	\norm{T^{-1}\hat{\Gamma}_{T,n}(f) - \int f \, d\pi}_{{\rm{L}}^2(\mathbb{P})} &\leq T^{-1} \norm{\hat{\Gamma}_{T,n}(f) - \Gamma_{T}(f)}_{{\rm{L}}^2(\mathbb{P})} + \norm{T^{-1} \Gamma_T(f) - \int f \, d\pi}_{{\rm{L}}^2(\mathbb{P})} \\
	&\leq \dfrac{C}{\sqrt{T}} \sqrt{\norm{f}_{\mathcal{D}^s(A)} \norm{f}_{\pi} \Delta_n} + \norm{T^{-1} \Gamma_T(f_0)}_{{\rm{L}}^2(\mathbb{P})}.
	\end{align*}
	We proceed to bound $\norm{T^{-1} \Gamma_T(f_0)}_{{\rm{L}}^2(\mathbb{P})}$, in which
	\begin{align*}
	\norm{T^{-1} \Gamma_T(f_0)}_{{\rm{L}}^2(\mathbb{P})}^2 &= 2 T^{-2} \int_0^T \int_0^h \langle P_{h-r}f_0,f_0 \rangle_{\pi}\,drdh \\
	&= \int_{\sigma(L)} \overline{\Phi}(\lambda) \,d\langle F_{\lambda} f_0,f_0 \rangle_{\pi},
	\end{align*}
	where $\overline{\Phi}$ is defined by, for $\lambda \in \sigma(L)$,
	\begin{align*}
	\overline{\Phi}(\lambda) &= 2 T^{-2} \int_0^T \int_0^h e^{\lambda(h-r)}\,drdh = 2 \dfrac{(\lambda T)^{-1}(e^{\lambda T} - 1) - 1}{\lambda T},
	\end{align*}
	and there exists a constant $\tilde{C}$ such that $|\overline{\Phi}(\lambda)| \leq \dfrac{\tilde{C}}{|\lambda|T}$. Using Lemma \ref{lem:dsa} gives
	\begin{align*}
	\norm{T^{-1} \Gamma_T(f_0)}_{{\rm{L}}^2(\mathbb{P})}^2 &\leq \dfrac{\tilde{C}}{T} \left|\int_{\sigma(L)} |\lambda|^{-1} \,d\langle F_{\lambda} f_0,f_0 \rangle_{\pi}\right| \\
	&\leq \dfrac{\tilde{C}}{T} \left(\int_{\sigma(L)} |\lambda|^{-2} \,d\langle H_{\lambda}^* f_0,f_0 \rangle_{\pi}\right) \norm{f_0}_{\pi} = \dfrac{\tilde{C}}{T} \norm{A^{-1}f_0}_{\pi} \norm{f_0}_{\pi}.
	\end{align*}	
	Finally, it follows from a standard change of measure argument to give \eqref{eq:esterror3}.

\section{Similarity orbit of reversible Markov chains}\label{sec:simorbit}

In this Section, our aim is to provide several illuminating examples for Theorem \ref{thm:mcins} and we will work in the continuous-time setting as this result generalizes easily to this setting, see Remark \ref{rk:cont} and \ref{rk:cont2}.  More precisely, suppose that we start with a reversible generator $G$ with transition semigroup $(Q_t)_{t \geq 0}$, we would like to characterize the family of Markov chains with generator $L$ associated with $G$ under the similarity transformation $G \Lambda = \Lambda L$ with $\Lambda$ being a bounded invertible Markov link. This idea allows us to generate Markov or contraction kernel from known ones in which the spectral decomposition, stationary distribution and eigenfunctions are linked by $\Lambda$. In addition, the so-called eigentime identity is preserved under intertwining as the spectrum is invariant under such transformation as stated in Theorem \ref{thm:mcins}. We will illustrate this approach by studying the  pure birth link in particular. While we consider univariate examples in the subsequent Section, nonetheless we can still handle the orbits of multivariate reversible Markov chains (e.g. \cite{KZ09,KM13,Griff16} and \cite{Zhou08}) by considering the link kernel to be the tensor product from univariate link and analyze the corresponding tensorized orbits.

Before detailing the examples, we introduce the following  notation that will be used throughout. Let $G$ be a reversible birth-death generator with respective to $\pi_G$ on $\mathcal{X} = \llbracket 0,\r \rrbracket$. Let
\begin{equation}
G(x,x-1) = \dx_x, \quad G(x,x) =-( \dx_x+ \bx_x) \quad \textrm{ and } \quad G(x,x+1) = \bx_x,
 \end{equation}
where $\dx_x$ (resp.~$\bx_x$) is the death (resp.~birth) rate at state $x$, and eigenvalues-eigenvectors denoted by $(-\lambda_j,\phi_j)_{j=0}^{N}$, where $\phi_j$ are orthonormal in $l^2(\pi_G)$.  We assume that $\dx_0 = \bx_{\r} = 0$. Write $(Q_t)_{t \geq 0}$ for  its  transition semigroup, then the spectral decomposition of $Q_t$ is given by
 \begin{equation}
 Q_t(x,y) = \sum_{j=0}^{\r} e^{-\lambda_j t}\phi_j(x) \phi_j(y) \pi_G(y).
 \end{equation}
 For further details on various birth-death models and their connections with orthogonal polynomials, we refer interested readers to \cite{KM59,Schoutens00,DKSC08,Sasaki09,KS96,Zhou08} and the references therein.

\subsection{Pure birth link on finite state space}

In this Section, we specialize into the case of $\mathcal{X} = \llbracket 0,\r \rrbracket$, with the link being the pure birth link as introduced by \cite{Fill} to study the distribution of hitting time and fastest strong stationary time, generated by birth-death processes with birth and death rates to be $b_x$ and $d_x$ respectively. The particular pure birth link $\Lambda_{pb}$ that we study is of the form
\begin{eqnarray}
  \Lambda_{pb}(x,y) &=& 1/2, x \in \llbracket 0,\r-1 \rrbracket, y \in \{x,x+1\}, \Lambda_{pb}(\r,\r) = 1, \\
  \Lambda_{pb}(x,y) &=& 0 \textrm{ otherwise}.
\end{eqnarray}
  A special feature in the pure birth orbit is that the heat kernel $P_t := e^{tL}$ of $L$ need not be Markovian, yet it still converges to $\pi_L$ exponentially fast as illustrated in Proposition \ref{prop:pb} below. Yet, we give sufficient conditions on a birth-death generator $G$ to guarantee $L$ to be Markov generator.

\begin{proposition}\label{prop:pb}
	 Suppose that $G \stackrel{\Lambda_{pb}}{\sim} L$  and denote by $(P_t)_{t \geq 0}$ being the transition semigroup associated with   $L$. Note that $(P_t)_{t \geq 0}$ need not be Markov under $\Lambda_{pb}$. For any $t \geq 0$ and $j,x,y \in \llbracket 0,\r \rrbracket$,  $P_t$ admits the following spectral decomposition
\begin{eqnarray*}
		 P_t(x,y) &=& \sum_{j=0}^{\r} e^{-\lambda_j t} f_j(x)f_j^*(y)\pi_L(y) ,
\end{eqnarray*}
where $f_j^*(y)\pi_L(y) = \frac{\phi_j(y-1)  \pi_G(y-1)}{2}\1_{y-1 \geq 0} + \phi_j(y) \pi_G(y)\left(\dfrac{\1_{y \neq \r}}{2} + \1_{y = \r}\right)$, $f_j(x) = \sum_{k=x}^{\r-1} (-2)^{k-x}\phi_j(k)+\phi_j(\r)$, and,
\begin{eqnarray} \label{eq:bd_bd}
	 ||P_t - \pi_L||_{TV} &\leq & \dfrac{\kappa(\Lambda_{pb}) e^{-\lambda_1t}}{2} \sqrt{\dfrac{1-\pi_{L}^*}{\pi_{L}^*}},
\end{eqnarray}
where $\pi_L(y) = \pi_G(y-1) \left(\dfrac{\1_{y-1 \geq 0}}{2} \right) +  \pi_G(y)\left(\dfrac{\1_{y \neq \r}}{2} + \1_{y = \r}\right)$ and recall that $ 	 \pi_L^* = \min_{y \in \llbracket 0,\r \rrbracket} \pi_L(y)$.

Moreover,  note that for all $x \in \llbracket 0,\r \rrbracket$,
$$L(x,x) = - \bx_x - \dx_{\max{\{x+1,\r\}}} \left(\dfrac{\1_{x+1 \geq \r}}{2} + \1_{x+1 < \r}\right)  < 0,$$
for $x > y+1$, $L(x,y) = 0$ and  for $x<\r$
	$$L(x+1,x) = \left(2 \1_{x+1 \neq \r} + \1_{x+1 =\r}\right) \dfrac{\dx_{x+1}}{2}  > 0.$$
If $\r\geq 4$, for $y \in \llbracket 1,\r-1 \rrbracket$,
$$L(y-1,y)= -\dx_{y-1} + \bx_y + \dx_{y+1}\1_{y < \r-1} +  \frac{\dx_{y+1}}{2}\1_{y+1=\r} \geq 0 \textrm{ and } L(\r-1,\r)=-\dx_{\r-1} + \bx_{\r-1} + \dx_{\r} \geq 0,$$
for $x \in \llbracket 0, y-2 \rrbracket$ and $y \in \llbracket 2,\r-1 \rrbracket$,
$$L(x,y)=(-1)^{x+y} \left( \bx_{y-2}+\dx_{y-1}-\bx_y-\dx_{y+1}+\frac{\dx_{\r}}{2}\1_{y=\r-1}\right) \geq 0,$$
and for $x \in \llbracket 0, \r-2 \rrbracket$
$$L(x,\r)=(-1)^{x+\r} \left(\bx_{\r-2} + \dx_{\r-1}-\bx_{\r-1}-\frac{\dx_{\r}}{2}\right) \geq 0,$$
then $L$ is a Markov generator.
\end{proposition}


\begin{rk}
	We can see that $\pi_L$ is the distribution at time $1$ of the Markov chain with transition matrix $\Lambda_{pb}$ under the initial law $\pi_G$.
\end{rk}

\begin{proof}
	We first observe that the inverse of $\Lambda_{pb}$ is given by
	\begin{eqnarray}\label{eq:inv_bd}
	\Lambda_{pb}^{-1}(x,y) &=& (-1)^{y-x}(2 \1_{y \neq \r} + \1_{y = \r})  \textrm{ for } x \leq y, x,y \in \llbracket 0,\r \rrbracket, \\
	\Lambda_{pb}^{-1}(x,y) &=& 0 \textrm{  otherwise}.
	\end{eqnarray}
	Upon expanding $P_t = \Lambda_{pb}^{-1} Q_t \Lambda_{pb}$, we get
	\begin{align*}
		P_t(x,y) &= \sum_{k=x}^\r (-1)^{k-x} (2 \1_{k \neq \r} + \1_{k = \r}) \left(Q_t(k,y-1) \left(\dfrac{\1_{y-1 \geq 0}}{2}\right) + Q_t(k,y) \left(\dfrac{\1_{y \neq \r}}{2} + \1_{y = \r}\right)\right) \\
		&= \sum_{j=0}^{\r} e^{-\lambda_j t} \left(\sum_{k=x}^{\r} (-1)^{k-x} (2 \1_{k \neq \r} + \1_{k = \r}) \phi_j(k)\right) \\
		&\quad \times \left(\phi_j(y-1)  \pi_G(y-1) \left(\dfrac{\1_{y-1 \geq 0}}{2} \right) + \phi_j(y) \pi_G(y)\left(\dfrac{\1_{y \neq \r}}{2} + \1_{y = \r}\right)\right),
	\end{align*}
	where the second equality follows from substituting the spectral expansion of $(Q_t)_{t \geq 0}$. The bound \eqref{eq:bd_bd}  follows directly from Corollary \ref{cor:spectralexp}. To show that $L$ is a Markov generator under the proposed conditions on birth and death rates, we need to impose sufficient conditions such that $L(x,x) < 0$ for all $x \in \mathcal{X}$ and $L(x,y) \geq 0$ for all $x \neq y \in \mathcal{X}$, see e.g. \cite[Chapter $20$]{LPW09}.  We proceed by calculating $G\Lambda_{pb}$, and the entries not mentioned below are all zero. We have
	\begin{eqnarray*}
	2G \Lambda_{pb}(x,x-1) &=&  \dx_x, \quad x \in \llbracket 1,\r \rrbracket, \\ 
	2G \Lambda_{pb}(x,x) &=&   -\bx_x -\dx_{\r}\1_{x=\r},\quad  x \in \llbracket 0,\r \rrbracket, \\
	2G \Lambda_{pb}(x,x+1) &=& -\dx_x +\bx_{\r-1}\1_{x=\r-1} ,  \quad  x \in \llbracket 0,\r-1 \rrbracket, \\
	2G \Lambda_{pb}(x,x+2) &=& \bx_x,  \quad  x \in \llbracket 0,\r-2 \rrbracket.
	\end{eqnarray*}

	Using the form of $G \Lambda_{pb}$ described above, we first note that $L(x,x) < 0$ is automatically satisfied since
	$$L(x,x) = - \bx_x - \dx_{\max{\{x+1,\r\}}} \left(\dfrac{1}{2}\1_{x+1 \geq \r} + \1_{x+1 < \r}\right)  < 0.$$ It remains to check $L(x,y) \geq 0$ for all $x \neq y$. Indeed, we have
	$$L(x,y) = \sum_{k=\max{\{x,y-2\}}}^{\min{\{y+1,\r\}}} (-1)^{k-x} (2 \1_{k \neq \r} + \1_{k = \r}) G\Lambda_{pb}(k,y).$$
	For $x > y+1$, $L(x,y) = 0$. For $x = y+1$,
	$$L(y+1,y) = \left(2 \1_{y+1 \neq \r} + \1_{y+1 =\r}\right) \dfrac{1}{2} \dx_{y+1} > 0.$$
	Thus, it boils down to check $L(x,y) \geq 0$ for $x \in \llbracket 0,y-1 \rrbracket$. For $y \in \llbracket 1,\r-1 \rrbracket$ and $x = y-1$,
	$$L(x,y) = -\dx_{y-1} + \bx_y + \dx_{y+1}\1_{y < \r-1} +  \frac{\dx_{y+1}}{2}\1_{y+1=\r}.$$
For $y = \r$ and $x = \r - 1$,
$$L(\r-1,\r) = 2G\Lambda_{pb}(\r-1,\r) - G \Lambda_{pb}(\r,\r) = -\dx_{\r-1} + \bx_{\r-1} + \dx_{\r}.$$
For $y \in \llbracket 2,\r-1 \rrbracket$ and $x \in \llbracket 0, y-2 \rrbracket$, since $\r\geq 4$,
	\begin{eqnarray*}
	 \nonumber 
	   L(x,y)&=& (-1)^{x+y} \left( 2G\Lambda_{pb}(y-2,y)-2G\Lambda_{pb}(y-1,y)+2G\Lambda_{pb}(y,y)-2G\Lambda_{pb}(y+1,y)\right) \\
	   &=& (-1)^{x+y} \left( \bx_{y-2}+\dx_{y-1}-\bx_y-\dx_{y+1}+\frac{\dx_{\r}}{2}\1_{y=\r-1}\right),
	\end{eqnarray*}
and for $y = \r$ and $x \in \llbracket 0, y-2 \rrbracket$,

\begin{eqnarray*}
	   L(x,\r)&=& (-1)^{x+\r} \left(2G\Lambda_{pb}(\r-2,\r)-2G\Lambda_{pb}(\r-1,\r)+G\Lambda_{pb}(\r,\r)\right) \\
	   &=& (-1)^{x+\r} \left(\bx_{\r-2} + \dx_{\r-1}-\bx_{\r-1}-\frac{\dx_{\r}}{2}\right).
	\end{eqnarray*}

\end{proof}


\begin{example}
	The pair
		$$G = \begin{pmatrix}
		-1 & 1 & 0 & 0 & 0 \\
		0.5 & -1 & 0.5 & 0 & 0 \\
		0 & 0.5 & -1 & 0.5 & 0 \\
		0 & 0 & 0.5 & -1 & 0.5 \\
		0 & 0 & 0 & 1 & -1
		\end{pmatrix}, \quad L = \begin{pmatrix}
		-1.5 & 1 & 0.5 & 0 & 0 \\
		0.5 & -1 & 0.5 & 0 & 0 \\
		0 & 0.5 & -1 & 0.5 & 0 \\
		0 & 0 & 0.5 & -1 & 0.5 \\
		0 & 0 & 0 & 0.5 & -0.5
		\end{pmatrix}$$
	satisfies the assumption Proposition \ref{prop:pb}, where $L$ is a non-reversible Markov generator since $\pi_L = (0.0625, 0.1875, 0.25,0.25,0.25)$ and $\pi_L(0) L(0,1) \neq \pi_L(1) L(1,0)$.
\end{example}
\begin{example}[Pure birth variants of constant rate birth-death processes with reflection at $0$ and $\r$]
	A more general example is that $\bx_x = \dx_x = \lambda$ for $x \in \llbracket 1,\r-1 \rrbracket$ and $\bx_0 =\dx_{\r} = 2 \lambda$ for some $\lambda > 0$. The stationary distribution $\pi_G$ is $\pi_G(x) = \frac{1}{\r}$ for $x \in \llbracket 1, \r-1 \rrbracket$ and $\pi_G(0) = \pi_G(\r) = \frac{1}{2\r}$, and the associated eigenvalues and orthogonal polynomials are, for $j,x \in \llbracket 0,\r \rrbracket$,
	\begin{align*}
		\lambda_j &= 2\lambda \left(1-\cos(\theta_j)\right) \\
		\phi_j(x) &= \cos(\theta_j x + c),
	\end{align*}
	where $(\theta_j)_{j=0}^{\r}$ and $c$ are determined by the boundary values $\cos(\theta x + c) = \cos(\theta) \cos(c)$ and $\cos(\theta(N-1) + c) = \cos(\theta) \cos(\theta N + c)$ and are arranged such that $(\cos(\theta_j))_{j=0}^{\r}$ is in non-increasing order, see \cite[Proposition $22$]{DM15}
	and \cite{Zhou08}. We proceed to check that the conditions in Proposition \ref{prop:pb} are fulfilled:
  for $y \in \llbracket 1,\r-1 \rrbracket$,
	$$L(y-1,y)=-\dx_{y-1} + \bx_y + \dx_{y+1}\1_{y < \r-1} + \frac{\dx_{y+1}}{2}\1_{y+1=\r} =  \lambda \1_{y < \r-1} + \frac{\lambda}{2}\1_{y+1=\r} \geq 0,$$
	and
	$$L(\r-1,\r)=-\dx_{\r-1} + \bx_{\r-1} + \dx_{\r} = 2\lambda \geq 0.$$
	For $x = 0$ and $y = 2$,
	\begin{equation*}
	  L(0,2) = \lambda,
	\end{equation*}
and otherwise for $y \in \llbracket 3,\r-1 \rrbracket$ and $x \in \llbracket 0, y-2 \rrbracket$,
	$$(-1)^{x+y} \left( \bx_{y-2}+\dx_{y-1}-\bx_y-\dx_{y+1}+\frac{\dx_{\r}}{2}\1_{y=\r-1}\right) = 0 \geq 0,$$
	and for $y = \r$ and $x \in \llbracket 0, y-2 \rrbracket$,
	$$(-1)^{x+\r} \left(\bx_{\r-2} + \dx_{\r-1}-\bx_{\r-1}-\frac{\dx_{\r}}{2}\right) = 0 \geq 0,$$
so $L$ is a Markov generator, with spectral decomposition given by
\begin{align*}
P_t(x,y) &= \sum_{j=0}^{\r} e^{-2\lambda \left(1-\cos(\theta_j)\right) t} f_j(x)f_j^*(y)\pi_L(y), \\
||P_t - \pi_L||_{TV} &\leq \dfrac{\kappa(\Lambda_{pb}) e^{-2\lambda \left(1-\cos(\theta_1)\right)t}}{2} \sqrt{\dfrac{1-\pi_{L}^*}{\pi_{L}^*}} = O(e^{-2\lambda \left(1-\cos(\theta_1)\right)t}), \quad \mathrm{where} \\
\pi_L^* &= \dfrac{1}{4\r}, \\
f_j(x) &= \sum_{k=x}^{\r} (-1)^{k-x} (2 \1_{k \neq \r} + \1_{k = \r}) \cos(\theta_j k + c) , \\
f_j^*(y)\pi_L(y) &= \cos(\theta_j (y-1) + c)  \pi_G(y-1) \left(\dfrac{\1_{y-1 \geq 0}}{2} \right) \\
&\quad+ \cos(\theta_j y + c) \pi_G(y)\left(\dfrac{\1_{y \neq \r}}{2} + \1_{y = \r}\right).
\end{align*}
\end{example}

\bibliographystyle{abbrvnat}

\end{document}